\documentclass[12pt,a4paper,oneside]{article}
\usepackage{arxiv}

\usepackage[utf8]{inputenc}
\usepackage{natbib}
\usepackage{graphicx}
\usepackage{placeins} 
\usepackage{float}

\usepackage{url}
\usepackage{subfig}
\usepackage{amsmath}
\usepackage{amsfonts}       
\usepackage{enumitem}
\usepackage{tabularx,booktabs} 
\usepackage{makecell}
\usepackage{multirow, makecell}

\usepackage{pdflscape}

\usepackage{wrapfig} 
\newcommand{\map}[3]{#1\colon #2 \to #3}   

\usepackage{amsthm}
\newtheorem{theorem}{Theorem}[section]
\newtheorem{corollary}[theorem]{Corollary}
\newtheorem{lemma}[theorem]{Lemma}
\newtheorem{proposition}[theorem]{Proposition}
\newtheorem{remark}[theorem]{Remark}
\newtheorem{definition}{Definition}
\newtheorem{assumption}{Assumption}

\usepackage{ulem}
\usepackage{mathtools}

\usepackage[x11names]{xcolor}

\usepackage{hyperref}

\title{Accelerated Gradient Dynamics on Riemannian Manifolds: Faster Rate and Trajectory Convergence}

\author{
    \textbf{Tejas Natu}$^\ast$\\
	 Mathematics and Computer Science\\
	Saarland University\\
	Germany
    \and
	\textbf{Camille Castera}\\
	Department of Mathematics\\
	University of Tübingen\\
	Germany
	\and
	\textbf{Jalal Fadili}\\
	ENSICAEN
	\\ Normandie Universit\'e
	\\CNRS, GREYC, France
	\and
	\textbf{Peter Ochs}\\
	 Mathematics and Computer Science\\
	Saarland University\\
	Germany
}

\makeatletter

\begin{document}

	\maketitle

	\begin{abstract}

        In order to minimize a differentiable geodesically convex function, we study a second-order dynamical system on Riemannian manifolds with an asymptotically vanishing damping term of the form \(\alpha/t\). For positive values of \(\alpha\), convergence rates for the objective values and convergence of trajectory is derived. We emphasize the crucial role of the curvature of the manifold for the distinction of the modes of convergence. There is a clear correspondence to the results that are known in the Euclidean case. When \(\alpha\) is larger than a certain constant that depends on the curvature of the manifold, we improve the convergence rate of objective values compared to the previously known rate and prove the convergence of the trajectory of the dynamical system to an element of the set of minimizers. For \(\alpha\) smaller than this curvature-dependent constant, the best known sub-optimal rates for the objective values and the trajectory are transferred to the Riemannian setting. We present computational experiments that corroborate our theoretical results.
	\end{abstract}

	\renewcommand*{\thefootnote}{$^\ast$}
	\footnotetext[1]{Corresponding author: \texttt{natu@math.uni-sb.com}}
	\renewcommand*{\thefootnote}{\arabic{footnote}}
	\setcounter{footnote}{0} 

    \section{Introduction}
    
    A perspective on constrained optimization problems that has gained substantial attention is the use of intrinsic geometry of the underlying space on which the optimization problem is posed. A typical problem is of the form 
    \begin{equation}\label{Problem1} \min_{x\in \mathcal{M}} f(x)\,, \end{equation}  
    
    where \(\mathcal{M}\) is a Riemannian manifold and \(\map{f}{\mathcal{M}}{\mathbb{R}}\) is geodesically convex. For the special case when \(\mathcal{M}=\mathbb{R}^{n}\), \eqref{Problem1} becomes a smooth convex unconstrained optimization problem. Several constrained optimization problems in different fields of science and engineering can be posed as optimization problems on Riemannian manifolds. This includes the eigenvalue problem \citep{golub2013matrix}, the Karcher-mean problem \citep{bini2013computing}, semidefinite programming \citep{burer2005local}, Gaussian Mixture Models \citep{hosseini2015matrix}, dictionary learning \citep{sun2016complete}, matrix completion \citep{vandereycken2013low} and statistical shape analysis \citep{ring2012optimization} among others. For a more comprehensive review, see \cite{boumal_2023}. Solving \eqref{Problem1} is challenging in general. Our interest in this paper is in the first-order methods (those using the Riemannian gradient of $f$) to solve \eqref{Problem1}. For instance, the Riemmanian gradient descent has been proposed and studied in detail, see e.g., \cite{udriste1994convex} and \cite{zhang2016first}. As a natural progression, more recently, first-order accelerated versions have also been studied on Riemannian manifolds, see e.g., \cite{ahn2020nesterov, alimisis2020continuous, alimisis2021momentum}. Motivated from the Euclidean setting, we aim to understand the (fast) convergence of first-order algorithms on problems posed over a manifold, such as \eqref{Problem1}. To this end, we take inspiration from Nesterov's accelerated gradient algorithm \citep{nesterov1983method}.
    
    For convex optimization problems in the Euclidean setting, that is, when the objective function \(f\) is a convex function on \( \mathbb{R}^{n}\), \citet{nesterov1983method} improved the rate of convergence of vanilla gradient descent from \(O\left(\frac{1}{k}\right)\) to \(O(\frac{1}{k^{2}})\), where \(k\) denotes the iteration number. In fact, this rate was proved to be optimal among all first-order methods for convex functions with Lipschitz conitnuous gradient \citep{Nesterov1}. This result is a milestone in the history of convex optimization and continues to be significant with the continually growing size and scale of practical problems. However, the analysis of Nesterov's method is non-trivial. As a result, several attempts have been made to understand the underlying mathematical structure of accelerated methods and come up with different perspectives to obtain new insights. One way of understanding acceleration is to look at the continuous-time dynamics of optimization algorithms, which provides powerful analytic tools. 
    
    \citet{su2014differential} proposed the following second-order in time ordinary differential equation (dynamical system) towards understanding Nesterov's accelerated gradient algorithm in the Euclidean case
    \begin{equation}\label{ODE1}  \ddot{X}(t) + \frac{\alpha}{t} \dot{X}(t) + \nabla f(X(t)) = 0\,,\end{equation} 
    
    for \(t> 0\) and \(\alpha>0\), with initial conditions \(X(0)=x_{0}\) and \(\dot{X}({0})= 0\).
    
    In this dynamical system, \(\dot{X}\) and \(\ddot{X}\) denote the velocity and acceleration of the trajectory \(X\) respectively. A curve \(\map{X}{[0,\infty)}{\mathbb{R}^{n}}\) such that \(X\in C^{2}(0,\infty)\cap C^{1}[0,\infty)\) is called a solution if it satisfies \eqref{ODE1} for \(t>0\) and the stated initial conditions. \cite{su2014differential} prove the existence and uniqueness of such a solution for \eqref{ODE1}.  
    
    A proper discretization of \eqref{ODE1} gives the Nesterov's accelerated gradient method. This system is similar to the classical spring--mass--damper system and equates force given as the product of unit mass and acceleration with the gradient of a convex potential function and the damping force proportional to the velocity and an asymptotically vanishing viscous damping coefficient \(\alpha/t\). In the absence of the damping term, \eqref{ODE1} is a conservative system and therefore the presence of damping is crucial for the system to be useful to solve an optimization problem. For \eqref{ODE1}, \citet{su2014differential} prove an accelerated convergence rate of \(f(X(t))-\mathrm{min}f = O\left(\frac{1}{t^{2}}\right)\) for \(\alpha\geq 3\) and this has triggered an immense follow up work in the Euclidean and Hilbertian setting. 
    
    In this vast array of works that follow \cite{su2014differential}, we focus on the contributions that form the basis of this work. In particular, \cite{may2017asymptotic} proves the convergence rate of objective values for \eqref{ODE1} with \(\alpha>3\) is strictly faster than \(O\big(\frac{1}{t^{2}}\big)\) and that \(f(X(t))-\mathrm{min}f = o(\frac{1}{t^{2}})\). In infinite dimensional real Hilbert spaces, \cite{attouch2018fast} show weak convergence of the trajectory to a point in the set of minimizers of \(f\) (\(\mathrm{argmin}f\)) provided the latter is non-empty. \cite{attouch2016rate} prove little-o convergence rate for objective values and the convergence of iterates in the discrete setting. \cite{attouch2019rate} analyze \eqref{ODE1} for the case \(0<\alpha \leq 3\). They show that the rate of convergence of objective values undergoes a phase transition and that \(\alpha=3\) is the smallest constant for which the rate of convergence of objective values is \(O\left(\frac{1}{t^{2}}\right)\). A similar analysis was carried out by \cite{vassilis2018differential} and \cite{apidopoulos2020convergence} in the case of differential inclusion problem modeling the FISTA algorithm and the Forward--Backward algorithm.
    
    It is natural to ask whether the above results can be extended to \eqref{Problem1} when \(\mathcal{M}\) is a Riemannian manifold. The continuous-time dynamical systems perspective for optimization on Riemannian manifolds has been studied by \cite{munier2007steepest} and \cite{alimisis2020continuous}. \cite{munier2007steepest} analyzed continuous-time dynamics of steepest descent method on Riemannian manifolds and proved convergence of the trajectory to a point in the set of minimizers for geodesically convex functions. \cite{alimisis2020continuous} generalized \eqref{ODE1} to Riemannian manifolds and proved \(f(X(t))-\mathrm{min}_{\mathcal{M}}f=O\left(\frac{1}{t^{2}}\right)\) when \(\alpha\) is chosen appropriately in a way that takes into account the curvature of the manifold (see Section \ref{PST} for a precise meaning). 
    
    Therefore, in line with a few related approaches we study a generalization of \eqref{ODE1} to Riemannian manifolds proposed by \cite{alimisis2020continuous} which we describe in detail in Section \ref{PST}. Our contributions are summarized as follows;
    
    \begin{enumerate}[label=(\roman*)]
        \vspace{-.20cm}
        \item When \(\alpha\) is larger than a threshold value, we prove the rate of convergence of objective values is actually \(o\left(\frac{1}{t^{2}}\right)\), which is strictly faster than the previously known rate \(O(\frac{1}{t^{2}})\). In addition, we prove the convergence of the trajectory to an element in the set of minimizers \(\mathrm{argmin}_{\mathcal{M}}f\). 
        
        \item  For \(\alpha\) below the threshold value, we provide convergence rates for objective values. In the same setting, we show convergence of trajectory to an element in \(\mathrm{argmin}_{\mathcal{M}}f\) under the condition that it satisfies the strong minimization property. 
        
        \item  We perform computational experiments that confirm our theoretical guarantees.  

    \end{enumerate}  
    
    \section{Related Work} \label{RelWo}

    In the Euclidean setting, some of the earlier works studying continuous-time dynamics for first-order accelerated methods include the works of \cite{alvarez2000minimizing} and \cite{attouch2000heavy} who study \eqref{ODE1} with a constant damping term instead of an asymptotically vanishing damping term. \cite{cabot2009long} study \eqref{ODE1} with a general asymptotically vanishing damping term \(\mathrm{a}(t)\) and showed that when \(\mathrm{a}(t)\) is non-integrable, the solution to the dynamical system possesses optimization properties i.e., \(f(X(t))\rightarrow \mathrm{min}f\). \cite{su2014differential} consider \(a(t)=\alpha/t\) and prove accelerated convergence rate of \(\mathcal{O}\left(\frac{1}{t^{2}}\right)\) when \(\alpha\geq 3\). For \(\alpha>3\), \cite{attouch2018fast} show convergence of the trajectory to an optimal solution while \cite{may2017asymptotic} shows little-o convergence rate for objective values. \cite{attouch2019rate} provide convergence results for the case \(0<\alpha\leq 3\) and prove convergence of trajectory in the case where the minimizer possesses strong minimization property. Further, in the case \(\alpha=3\), for a convex objective function, they show convergence of trajectory to the optimal solution in one dimensional problems. More recently, \cite{attouch2022ravine} have studied the Ravine method from a dynamical systems perspective and drawn similarities with the Nesterov's accelerated gradient method. A more comprehensive survey of historical aspects and research trends related to continuous-time dynamics for achieving acceleration is provided by \cite{attouch2022ravine} and \cite{attouch2017asymptotic}. 
    
     In the Riemannian setting, standard references on optimization algorithms on Riemannian manifolds include \cite{absil2008optimization} and \cite{boumal_2023}. \cite{udriste1994convex} is a standard reference for convex analysis on Riemannian manifolds while \cite{vishnoi2018geodesic} provides a detailed pedagogical survey of geodesically convex sets and geodesically convex functions on Riemannian manifolds. \cite{zhang2016first} develop techniques that are used to provide convergence guarantees for gradient descent method for geodesically convex functions on Riemannian manifolds. First-order accelerated algorithms for minimizing geodesically convex functions on Riemannian manifolds have been studied by \cite{zhang2018towards,ahn2020nesterov,han2023riemannian,alimisis2021momentum}. \cite{zhang2018towards} propose a computationally tractable accelerated method for geodesically strongly convex functions by proposing a new estimate sequence and show accelerated rate of convergence locally. \cite{ahn2020nesterov} propose the first global accelerated algorithm on Riemannian manifolds for geodesically strongly convex functions. In particular, \cite{zhang2018towards} and \cite{ahn2020nesterov} develop techniques to tackle metric distortion that is inherent to the analysis of algorithms on Riemannian manifolds. \cite{han2023riemannian} generalize the work of \cite{scieur2016regularized} to Riemannian manifolds and propose acceleration using extrapolation. \cite{alimisis2021momentum} employ momentum in combination with techniques developed by \cite{zhang2016first} to achieve acceleration and provide accelerated convergence guarantees for geodesically convex functions. 

     A peculiar aspect of the analysis of accelerated first-order methods on Riemannian manifolds is the set of assumptions under which the results are proved. In general, one of the standard conditions in Riemannian optimization is to assume that the exponential map is a global diffeomorphism (see Section \ref{SecPrelim} for further details). This condition ensures that the exponential map is invertible and smooth. Additionally, in order to derive results about first-order accelerated methods in both the discrete and continuous setting, we work in a bounded subset of the manifold and in particular, we must make the assumption that the trajectories in the continuous setting or iterates in the discrete setting lie in that bounded domain, see e.g., \cite{zhang2018towards}, \cite{alimisis2020continuous} and \cite{alimisis2021momentum}. In other words, this means that the results are valid only for trajectories or iterates that lie within that bounded subset. This is because in curved spaces, the analysis makes use of certain comparison theorems like the Rauch comparison theorem \citep{petersen2006riemannian} that dictate the size of the domain in which we must confine our analysis. 
     
     In this work, we consider the continuous-time dynamical system approach towards understanding acceleration of first-order optimization methods on Riemannian manifolds. In particular, we close the gaps in convergence guarantees between the Euclidean and Riemannian settings for the continuous-time dynamics modelling Nesterov’s acceleration. We shall work with similar assumptions and for reasons discussed above. A more detailed description follows in later sections.    
    
    \section{Preliminaries from Riemannian Geometry} \label{SecPrelim}
    
    We recall some basic concepts from Riemannian geometry that we shall make references to during the course of this work. This material can be found in standard references on the subject like \cite{Tu1}, \cite{Godinho1} and \cite{boumal_2023}.
    
    \paragraph{Riemannian manifolds.}A smooth manifold is a Hausdorff, second-countable topological manifold such that the chart transition maps are of class \(C^{\infty}\). To a smooth manifold, it is possible to attach at every point \(x\), a real vector space called the tangent space \(T_{x}\mathcal{M}\). The union of all tangent spaces over the manifold \(\mathcal{M}\) can be imparted a smooth manifold structure and is called the tangent bundle \(T\mathcal{M}=\cup_{x\in\mathcal{M}}T_{x}\mathcal{M}\). A smooth vector field on the manifold is a smooth map \(\map{Z}{\mathcal{M}}{T\mathcal{M}}\). A Riemannian manifold is an ordered pair \(\left(\mathcal{M},\langle \cdot,\cdot \rangle\right)\) where \(\langle\cdot,\cdot\rangle\) defines an inner product \(\langle\cdot,\cdot\rangle_{x}\) on the tangent space \(T_{x}\mathcal{M}\) for every \(x\in \mathcal{M}\). This assignment is smooth in the sense that the map \(x\mapsto \langle Y_{x},Z_{x}\rangle_{x}\) is a \(C^{\infty}\) function on \(\mathcal{M}\), where \(Y_{x}\) and \(Z_{x}\) are tangent vectors at \(x\) corresponding to \(C^{\infty}\) vector fields \(Y\) and \(Z\) on the manifold. The inner product on the tangent space gives the norm of a tangent vector \(Y_{x}\) as \(\|Y_{x}\|_{x}\coloneqq\sqrt{\langle Y_{x},Y_{x}\rangle}_{x}\). From now on, the subscript on the inner product highlighting the point on which the inner product is evaluated will be dropped when it is clear from the context. 
    
    The Riemannian inner product allows for measurement of length of a piece-wise smooth curve \(\map{\gamma}{[a,b]}{\mathcal{M}}\) using the formula \(\ell(\gamma)\coloneqq\int_{a}^{b}\|\gamma^{\prime}(t)\|_{\gamma(t)}\), where \(\gamma^{\prime}(t)\) is the velocity vector field of the curve \(\gamma\). This gives rise to the notion of distance between two points \(x\) and \(y\) on the manifold \(\mathcal{M}\) given as \(d(x,y)\coloneqq\inf_{\gamma}\ell(\gamma)\), where infimum is taken over all piecewise smooth curves from \(a\) to \(b\) on \(\mathcal{M}\) such that \(\gamma(a)=x\) and \(\gamma(b)=y\). The Riemannian manifold equipped with this distance becomes a metric space. We can also define the diameter of a subset \(\mathcal{C}\) of the manifold \(\mathcal{M}\) as \(\mathrm{diam}(\mathcal{C}) = \mathrm{sup}_{x,y\in\mathcal{C}}d(x,y)\). 
    
    \paragraph{Geodesics and parallel transport.}Geodesics generalize the notion of straight lines on curved spaces. The differentiation of vector fields along a curve is possible via the notion of a covariant derivative \(\frac{D}{dt}\) associated with the unique connection operator on a Riemannian manifold called the affine or Levi--Civita connection \(\nabla\). Given a curve \(\gamma\) on a manifold \(\mathcal{M}\) and the corresponding velocity vector field \(\gamma^{\prime}\). Assuming that \(\gamma\) is at least \(C^2\)-- smooth, we call \(\gamma\) a geodesic if its velocity vector is constant, i.e. \( \frac{D}{dt}\big(\frac{d\gamma}{dt}\big)=0 \). Geodesics can also be defined as the solution to the variational problem of finding the curve of shortest length between two points. It should be noted that, a curve with the least distance between two points has zero acceleration, however a curve with zero acceleration need not be curve of least distance between two points. For example, on a sphere any two points can joined by both the short and long segments of the same geodesic which is the great circle.
    
    Parallel transport refers to transporting a tangent vector along a curve such that it remains constant along the curve. On a Riemannian manifold \(\mathcal{M}\) equipped with its affine connection and the covariant derivative \(\frac{D}{dt}\), for any smooth curve \(\gamma\) and \(v\in T_{\gamma(0)}\mathcal{M}\), there exists a unique vector field \(Z\) along the curve \(\gamma\) such that \(\frac{D}{dt}Z=0\) and \(Z(0)=v\) \citep{boumal_2023}. We use the notation by \cite{zhang2016first} to denote parallel transport by \(\Gamma_{x}^{y}v\) where \(v\in T_{x}(\mathcal{M})\) is transported to \(y\), that is, \(T_{y}(\mathcal{M})\) via the geodesic \(\gamma\). Parallel transport preserves inner products, i.e. \(\langle u,\,v\rangle_{x} = \langle \Gamma_{x}^{y}u,\,\Gamma_{x}^{y}v\rangle_{y}\).

    The notion of parallel transport allows us to define \(L\)- smoothness of a function \(f\) defined on a Riemannian manifold \(\mathcal{M}\). A function \(f\) is geodesically \(L\)- smooth if \(\left\|\mathrm{grad}f(x)-\Gamma_{y}^{x}\mathrm{grad}f(y)\right\|_{x}\leq L\ell(\gamma)\) for some \(L>0\) and for all \(x,y\in \mathcal{M}\) and a geodesic \(\gamma\).
    
     \paragraph{Exponential and logarithmic maps.}The exponential map denoted as \(\map{\mathrm{Exp}_{x}}{T_{x}\mathcal{M}}{\mathcal{M}}\), operates on a tangent vector \(v\in T_{x}\mathcal{M}\) and gives a point on the manifold that lies on the unique geodesic through \(x\) with initial velocity \(v\). The point \(\mathrm{Exp}_{x}(v)\) lies at a distance of \(\|v\|_{x}\) from \(x\) on the geodesic. The exponential map is not injective, however on a Riemannian manifold, we can define the radius of injectivity where the exponential map is a diffeomorphism. If the radius of injectitivity is non-zero, then within this neighborhood, it is possible to define the inverse of the exponential map called as logarithmic map denoted as \(\map{\mathrm{Log}_{x}}{\mathcal{M}}{T_{x}\mathcal{M}}\). \(\mathrm{Log}_{x}(y)\) gives the tangent vector in \(T_{x}(\mathcal{M})\) whose exponential map would give \(y\) and whose length equals the distance \(d(x,y)\), i.e. \(d(x,y)=\left\|\mathrm{Log}_{x}y\right\|_{x}\).
    
    \paragraph{Riemannian gradient and Riemannian Hessian.}For a smooth function \(\map{f}{\mathcal{M}}{\mathbb{R}}\), the differential \(\map{Df(x)}{T_{x}\mathcal{M}}{\mathbb{R}}\) is defined as \(Df(x)[v]\coloneqq(f\circ \gamma)^{\prime}(0)\), where \(v\in T_{x}\mathcal{M}\) and \(\gamma\) is a curve on the manifold such that \(\gamma(0)=x\) and \(\gamma^{\prime}(0)=v\). The Riemannian gradient of \(f\) is the unique vector field denoted by \(\mathrm{grad}f\) on \(\mathcal{M}\) such that for all \((x,v)\in T\mathcal{M}\) we have \(Df(x)[v]= \langle v,\,\mathrm{grad}f(x)\rangle_{x}\). 
    The Riemannian Hessian of \(f\) at \(x\in\mathcal{M}\) is a linear operator \(\map{\mathrm{Hess}f}{T_{x}\mathcal{M}}{T_{x}\mathcal{M}}\) defined as \(\mathrm{Hess}f(x)[v] = \nabla_{v}\mathrm{grad}f\).
    
    \paragraph{Sectional curvature.}Sectional curvature generalizes the notion of Gaussian curvature of two-dimensional surfaces to higher dimensions. Starting with any two-dimensional subspace \(\Pi_{p}\) of the tangent space at a point \(x\in\mathcal{M}\), the image of \(\Pi_{x}\) under the exponential map locally spans a two-dimensional surface \(S_{\Pi_{x}}\) such that \(T_{x}S_{\Pi_{x}}=\Pi_{x}\). Then the sectional curvature denoted as \(K(\Pi_{x})\) associated with \(\Pi_{x}\) is the Gaussian curvature of \(S_{\Pi_{x}}\). Since the sectional curvature is dependent of the choice of the subspace \(\Pi_{p}\), we work with a tight global lower bound on the sectional curvature of the manifold denoted by \(K_{\mathrm{min}}\). A similar tight global upper bound on the sectional curvature is denoted by \(K_{\mathrm{max}}\).
    
    \paragraph{Geodesic convexity.}The notion of convex sets and convex functions can be generalized to Riemannian manifolds by replacing straight lines with geodesics. A subset \(\mathcal{C}\subset \mathcal{M}\) is called a geodesically convex set if for every \(x,y\in \mathcal{C}\), there exists a geodesic \(\map{\gamma}{[0,1]}{\mathcal{M}}\) such that \(\gamma(0)=x\) and \(\gamma(1)=y\) and \(\gamma(t)\in \mathcal{C}\) for \(t\in[0,1]\). Since it is not necessary to have a unique geodesic between two points on a manifold (for example, two points on a sphere are joined by two segments of the same geodesic great circle of different lengths), we can define geodesically unique convex sets. A geodesically unique convex set has one geodesic segment joining any two points in the set. A function \(f:\mathcal{C}\rightarrow \mathbb{R}\) is called geodesically convex function if for any geodesic \(\gamma\) with \(\gamma(0)=x\) and \(\gamma(1)=y\) we have \(f(\gamma(t))\leq (1-t)f(x)+tf(y)\,\text{for all}\,t\in[0,1]\). Further, a differentiable geodesically convex function satisfies 
    \begin{equation}\label{SJ2} f(y)\geq f(x)+\left\langle\mathrm{grad}f(x),\,\mathrm{Log}_{x}y\right\rangle_{x}\,,\end{equation} 

    for every \(x\) and \(y\) in the geodesically uniquely convex set \(\mathcal{C}\). We shall refer to \eqref{SJ2} through out this work. 
    
    \section{Problem Setting}\label{PST}
     
    We consider the problem in \eqref{Problem1} and study the following dynamical system proposed by \cite{alimisis2020continuous} as a generalization to \eqref{ODE1} 
    \begin{equation}\label{RODE1} \nabla \dot{X}(t) + \frac{\alpha}{t}\dot{X}(t) + \mathrm{grad} f(X(t)) = 0\,,\,\, t>0\,,\quad X(0)=x_{0}\,\, \text{and} \,\,\dot{X}(0)=0\,\,,\end{equation} 
    for \(\alpha>0\) and \(x_{0}\in \mathcal{M}\). 
    
    In this dynamical system, \(\mathrm{grad}f\) denotes the Riemannian gradient of the objective function \(f\) in \eqref{Problem1}, \(\dot{X}\) denotes the velocity vector field of the trajectory \(X\) and \(\nabla\dot{X}\) denotes the covariant derivative of the velocity vector field that generalizes the acceleration term in \eqref{ODE1}.
    
   \begin{definition}A curve \(X\in C^{2}(0,\infty)\cap C^{1}[0,\infty)\) on the manifold \(\mathcal{M}\) that satisfies \eqref{RODE1} is called a solution to \eqref{RODE1}.
       
   \end{definition} \cite{alimisis2020continuous} prove existence of solution for \eqref{RODE1} with \(\alpha>0\) under conditions stated in Assumption \ref{Assumption1} below. Their existence result is stated in Proposition \ref{Propo1}. However, uniqueness of solution to \eqref{RODE1} is not guaranteed.
   
    \begin{assumption}\text{}
    \label{Assumption1}
    \vspace{-0.25cm}
        \begin{enumerate}
            \item[i)] The objective function \(f\) is geodesically convex and geodesically \(L\)-smooth.
            \item[ii)] The manifold \(\mathcal{M}\) is geodesically complete.
            \item[iii)] The exponential map is a global diffeomorphism on the manifold \(\mathcal{M}\). 

        \end{enumerate}
    \end{assumption}
    
    \begin{remark}
    For a geodesically complete Riemannian manifold \(\mathcal{M}\), any two points on \(\mathcal{M}\) can be joined by a geodesic. The assumption of a global diffeomorphism of the exponential map ensures that the logarithmic map is well defined. Further, it is worth mentioning that an important class of manifolds called the Hadamard manifolds satisfy these conditions \citep{petersen2006riemannian}.
    \end{remark} 
    \begin{proposition}\label{Propo1}\citep{alimisis2020continuous} Under Assumption \ref{Assumption1},  for \(\alpha>0\), System \ref{RODE1} has a solution \(\map{X}{[0,\infty)}{\mathcal{M}}\).
    \end{proposition}
    
    In this work, we perform a thorough study of the asymptotic behavior of solutions to \eqref{RODE1}. For this, the curvature of the manifold plays a key role. A major difference between the Euclidean and the Riemannian setting is that the curvature of the manifold determines the choice of \(\alpha\) and hence the convergence rates. To explore this, we first discuss a crucial geometric result provided by \cite{alimisis2020continuous}. 

    Let \(K_{\mathrm{max}}\) and \(K_{\mathrm{min}}\) be the upper and lower bounds on the sectional curvature of \(\mathcal{M}\) as discussed in Section \ref{SecPrelim} and fix a diameter \(D\) that satisfies the following, 
    \begin{equation}\label{SJ1} D<\frac{\pi}{\sqrt{K_{\mathrm{max}}}}\,,\,\text{if}\,K_{\mathrm{max}}>0\,, \quad\text{and}\quad D < \infty\,,\,\text{if} \,K_{\mathrm{max}}\leq 0\,. \end{equation} 
    
    Consider a subset \(\mathcal{C}\subset\mathcal{M}\), such that \(\mathrm{diam}(C)\leq D\) where \(\mathrm{diam}(C)\) denotes the diameter of \(C\) (as discussed in Section \ref{SecPrelim}), a curve \(\map{X}{I}{\mathcal{C}}\), where \(I\subset \mathbb{R}\) and a point \(z\in \mathcal{C}\). Then \(d\left(X(t),z\right)\) quantifies the distance between a point \(X(t)\) on the curve \(X\) and the point \(z\). A key step in the analysis requires a bound on the eigenvalues of \(-\mathrm{Hess}\left(-\frac{1}{2}d\left(X(t),z\right)^{2}\right)\), where \(\mathrm{Hess}\) denotes the Riemannian Hessian. This is equivalent to an expression of the form \(\left\langle-\nabla_{\dot{X}}\mathrm{grad}\left(-\frac{1}{2}d\left(X(t),z\right)^{2}\right), \dot{X}\right\rangle\) where \(\nabla_{\dot{X}}\mathrm{grad}\left(-\frac{1}{2}d\left(X(t),z\right)^{2}\right)\) denotes the covariant derivative of the Riemannian gradient vector field of \(-\frac{1}{2}d\left(X(t),z\right)^{2}\). The Riemannian gradient vector field is given as \begin{equation}\label{last1}\mathrm{grad}\left(-\frac{1}{2}d\left(X(t),z\right)^{2}\right)=\mathrm{Log}_{X(t)}z\,,   \end{equation}
    
    and a proof for \eqref{last1} can be found in \citep{pennec2018barycentric,alimisis2020continuous}.
    
    Then \cite{alimisis2020continuous} provide the following bounds, 
    \begin{equation}\label{OMG1} \sigma\left(d\left(X(t),z\right)\right)\left\| \dot{X}(t)\right\|^{2} \leq \left\langle \nabla_{\dot{X}(t)}\mathrm{Log}_{X(t)}z, -\dot{X}(t) \right\rangle \leq \xi\left(d\left(X(t),z\right)\right)\left\| \dot{X}(t) \right\|^{2}\,, \end{equation}
    
    where 
    \begin{equation}\label{OMG2}\sigma\left(p\right) \coloneqq \begin{cases}1\,\,\,\,\,\,\,\,\,\,\,\,\,\,\,\,\,\,\,\,\,\,\qquad\qquad\qquad,\,\text{if}\,K_{\mathrm{max}}\leq 0\,; \\ \sqrt{K_{\mathrm{max}}}\,p\,\mathrm{cot}\left(\sqrt{K_{\mathrm{max}}}\,p\right),\,\text{if}\,K_{\mathrm{max}}>0\,,\end{cases}\end{equation}
    
    and 
    \begin{equation}\label{OMG3}\xi\left(p\right) \coloneqq \begin{cases}\sqrt{-K_{\mathrm{min}}}\,p\,\mathrm{coth}\left(\sqrt{-K_{\mathrm{min}}}\,p\right)\,,\,\text{if}\,K_{\mathrm{min}}<0\,; \\ 1\,\,\,\,\,\,\,\,\,\,\,\,\,\,\,\,\,\,\,\,\,\qquad\qquad\qquad\qquad,\,\text{if}\,K_{\mathrm{min}}\geq 0\,.\end{cases}\end{equation}

   \begin{figure}[h!t]
    \centering
    \includegraphics[width=5.25in]{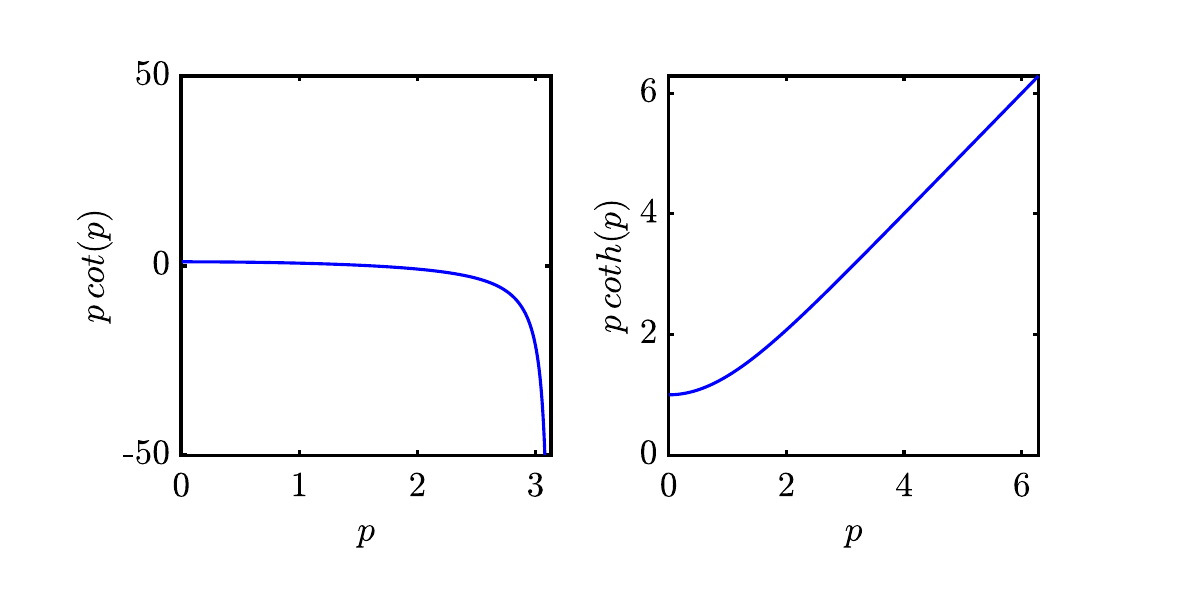}
    \caption{Functions used in the bounds given in \eqref{OMG1} with \(K_{\mathrm{max}}=1\) and \(K_{\mathrm{min}}=-1\).}
    \label{figg1}
    \end{figure}

    The functions \(p\,\mathrm{cot}(p)\) and \(p\,\mathrm{coth}(p)\) are visualized in Figure \ref{figg1}. The terms \(\sigma\left(d\left(X(t),z\right)\right)\) and \(\xi\left(d\left(X(t),z\right)\right)\) are bounds on the eigenvalues of the operator \(-\mathrm{Hess}\left(-\frac{1}{2}d\left(X(t),z\right)^{2}\right)\). From \eqref{OMG2} and \eqref{OMG3} we observe that the curvature of the manifold impacts these bounds. For the analysis of \eqref{RODE1}, the choice of \(\alpha\) depends on the upper bound on the eigenvalues and this aspect becomes clear from the proof of Theorem \ref{pythagoras} hereafter. Now, since \(\xi\left(d\left(X(t),z\right)\right)\) is dependent on the parameter \(t\), we instead consider an upper bound on \(\xi\) by evaluating it at \(D\) and define:
    \begin{equation}\label{Noidea1} \zeta \coloneqq \xi\left(D\right) \quad \text{and} \quad \delta \coloneqq 2\zeta+1\,. \end{equation}
    
    We now make some important observations about the terms \(\zeta\) and \(\delta\).  

    \begin{enumerate}[label=(\roman*)]
    
        \item Since the function \(p\coth(p)\) is strictly increasing for \(p\in (0,\infty)\), on the set \(\mathcal{C}\), we have \(\xi\left(d\left(X(t),z\right)\right)\leq \zeta\). Thus the upper bound in \eqref{OMG1} can be bounded as 
        \begin{equation}\label{Noidea3}\left\langle \nabla_{\dot{X}(t)}\mathrm{Log}_{X(t)}z,-\dot{X}(t)\right\rangle \leq \xi\left(d\left(X(t),z\right)\right)\left\| \dot{X}(t) \right\|^{2} \leq \zeta\left\|\dot{X}(t)\right\|^{2}\,.
        \end{equation} 
        \item When \(K_{\mathrm{min}}\geq 0\), from \eqref{OMG3} we have \(\zeta=1\) and thus \(\delta=3\). 
        \item When \(K_{\mathrm{min}} < 0\), from \eqref{OMG3}, since \(p\,\mathrm{coth(p)>1}\) for \(p\in (0,\infty)\), we have \(\zeta>1\) and thus \(\delta >3\).

    \end{enumerate}  
    
     At this stage, we summarize the chain of events. We begin with a second-order system as defined in \eqref{RODE1} and under Assumption \ref{Assumption1}, we have the existence of solutions to the system for \(\alpha>0\). We calculate \(\delta\) as in \eqref{Noidea1} and this is independent of any conditions required for the existence of a solution. 
     This is important to note because later we will analyze \eqref{RODE1} for \(0<\alpha\leq \delta\) and \(\alpha>\delta\) which is the Riemannian analog of \(0<\alpha\leq 3\) and \(\alpha>3\) in the Euclidean setting. So in the Riemannian setting, \(\delta\) corresponds to the constant \(3\) in the Euclidean case. Finally, to prove our main results we will make use of the bound in \eqref{Noidea3} for the case where \(X\) is any solution of \eqref{RODE1} and \(z\in \mathrm{argmin}_{\mathcal{M}}f\). 
     
    Based on this discussion, we complement Assumption \ref{Assumption1} with the following standing assumptions. 
    
    \begin{assumption} Let \(D\) satisfy \eqref{SJ1}.
    \label{Assumption2}
    \vspace{-0.25cm}
        \begin{enumerate} 
            \item[i)] The sectional curvature of \(\mathcal{M}\) is lower bounded by \(K_{\min}>-\infty\). 
            \item[ii)] \(\mathcal{C}\) is a geodesically convex subset of \(\mathcal{M}\) with \(\mathrm{diam}(C)\leq D\). 
            \item[iii)] The set of minimizers \(\mathrm{argmin}_{\mathcal{M}}f\neq \emptyset\) and \(\mathrm{argmin}_{\mathcal{M}}f\subset \mathcal{C}\).
            \item[iv)] The initial point \(x_{0}\in\mathcal{C}\) and all the solutions to \eqref{RODE1} remain inside the set \(\mathcal{C}\).
        \end{enumerate}
    \end{assumption}

    \begin{remark}
     In order to use \eqref{SJ2}, we make the assumption that \(\mathcal{C}\) is geodesically convex. Since uniqueness of solution to \eqref{RODE1} is not guaranteed, we make the assumption that all trajectories remain inside \(\mathcal{C}\). Furthermore, we have a rather mild assumption that the set of minimizers is contained in \(\mathcal{C}\). The last condition in Assumption \ref{Assumption2} has been discussed in Section \ref{RelWo} as a standard assumption in the study of first-order accelerated dynamics and algorithms on Riemannian manifolds, see e.g., \cite{zhang2018towards} and \cite{alimisis2021momentum} (in the discrete setting) and \cite{alimisis2020continuous} (in the continuous-time setting).
    \end{remark} 
    
    For \eqref{RODE1} with \(\alpha=\delta\), under Assumptions \ref{Assumption1} and \ref{Assumption2}, \cite{alimisis2020continuous} prove that the convergence rate of objective values satisfies \(f(X(t))-\mathrm{min}_{\mathcal{M}}f=O\left(\frac{1}{t^{2}}\right)\). In this work, we extend the analysis by providing faster convergence rates and the convergence of solution to the set of minimizers for the case when \(\alpha>\delta\). We complete the analysis by providing convergence guarantees for the case when \(0<\alpha\leq\delta\).
    
    Thus, from the discussion in this section, we observe that the curvature of a Riemannian manifold impacts the choice of the damping coefficient \(\alpha\) via a curvature-dependent term given by \(\delta\). We now present the main results of this work.
    
    \section{Main Results}\label{MR}
    
    In the case \(\alpha >\delta\), which is the Riemannian analog of \(\alpha>3\) in the Euclidean setting, we improve the convergence rate for objective values in Theorem \ref{pythagoras} and prove the convergence of solution trajectories of \eqref{RODE1} to an element in the set \(\mathrm{argmin}_{\mathcal{M}}f\) in Theorem \ref{LeviCivita}. In Theorem \ref{Fermat}, we analyze the convergence rate in the sub-critical case \(0< \alpha\leq \delta\) and for the same setting, in Theorem \ref{Chandrama}, we prove convergence of trajectories under the assumption that the minimizer satisfies the strong minimization property. 
    
    \subsection{Improved convergence rate when \(\alpha>\delta\)}

    Our first result improves the rate of convergence of objective values from \(O\big(\frac{1}{t^{2}}\big)\) to \(o\big(\frac{1}{t^{2}}\big)\). We consider \eqref{RODE1} and perform a Lyapunov analysis similar to \citet{may2017asymptotic, attouch2018fast} and prove the following result.
    
    \begin{theorem}\label{pythagoras}
    
    Assume \(\alpha>\delta\) in \eqref{RODE1}. Then under Assumptions \ref{Assumption1} and \ref{Assumption2}, any solution \(X\) of \eqref{RODE1} satisfies  
    \begin{equation*}f(X(t))-\mathrm{min}_{\mathcal{M}} f = o\left(\frac{1}{t^{2}}\right)\,.\end{equation*}
    
    \end{theorem}
    
    \begin{proof}
    
    We fix some \(z\in\mathrm{argmin}_{\mathcal{M}}f\), define \(f^{\star}\coloneqq \mathrm{min}_{\mathcal{M}}f\) and introduce the following functions \(\map{W,\,h}{[t_{0},\infty)}{[0,\infty)}\) as 
    \begin{align*}\label{oo5}
     W(t) \coloneqq \frac{1}{2}\left\|\dot{X}(t)\right\|^{2} + f(X(t))-f^{\star}\quad\text{and}\quad
     h(t) \coloneqq \frac{1}{2}d\left(X(t),z\right)^{2}\,,
    \end{align*} 
    
    The proof strategy consists of the following steps.
    
    \begin{enumerate}[label=(\roman*)]
    
        \item We show that \(W^{\prime}(t)\leq 0\) which shows that \(W(t)\) is a non-increasing function.
        \item We show that \(\lim_{t\to\infty} t^{2}W(t)\) exists and is equal to some \(m\geq 0\). This establishes big--O convergence rate.
        \item We show that \(\int_{t_{0}}^{\infty}sW(s)ds<\infty \)\,.
        \item Based on a simple lemma described in Appendix \ref{lema1}, we must have \(m=0\).
        \item Since \(W(t)\) is a sum of positive quantities, we deduce in particular that
        \begin{equation*} \lim_{t\to\infty} t^{2}\big[f\left(X(t)\right) - f^{\star}\big]=0\,,\end{equation*}  
        
        which gives us our result. 
    \end{enumerate} 

    We now provide details of the proof of each step for which we will need the first and second derivatives of \(h\). The derivatives of \(h\) are calculated using properties of covariant derivatives of smooth vector fields on manifolds and are given as 
    \begin{align}\label{Prelim3} h^{\prime}(t) &= \left\langle \mathrm{Log}_{X(t)}z,-\dot{X}(t)\right\rangle\,, \\\label{Prelim4} h^{\prime\prime}(t) &= \left\langle \nabla_{\dot{X}(t)}\mathrm{Log}_{X(t)}z,-\dot{X}(t)\right\rangle + \left\langle \mathrm{Log}_{X(t)}z,-\nabla \dot{X}(t)\right\rangle\,. \end{align} 
    
    A proof for \eqref{Prelim3} can be found in \cite{alimisis2020continuous} whereas \eqref{Prelim4} follows from the product rule for covariant derivatives, see e.g., \cite{Tu1}[Theorem 13.2]. An explanation of these expressions for the derivatives of \(h\) including a comparison with the Euclidean case can be found in Appendix \ref{TH1}.
    
    Using the abbreviation \(\kappa(t)=\frac{\alpha}{t}\), we can write
    \begin{align} \begin{split}\label{eqn1001}
    h^{\prime\prime}(t) + \kappa(t) h^{\prime}(t) &= \left\langle \nabla_{\dot{X}(t)}\mathrm{Log}_{X(t)}z,-\dot{X}(t)\right\rangle + \left\langle \mathrm{Log}_{X(t)}z,-\nabla \dot{X}(t)\right\rangle + \kappa(t) \left\langle \mathrm{Log}_{X(t)}z,-\dot{X}(t)\right\rangle \\ & = \left\langle \nabla_{\dot{X}(t)}\mathrm{Log}_{X(t)}z,-\dot{X}(t)\right\rangle + \left\langle \mathrm{Log}_{X(t)}z,-\nabla\dot{X}(t)-\kappa(t)\dot{X}(t)\right\rangle \\&= \left\langle \nabla_{\dot{X}(t)}\mathrm{Log}_{X(t)}z,-\dot{X}(t)\right\rangle + \left\langle \mathrm{Log}_{X(t)}z,\mathrm{grad}f(X(t))\right\rangle\,, \end{split}\end{align}
    
    where the last equality follows from the definition of the ODE in \eqref{RODE1}. Next, we have 
    \begin{align}\begin{split}\label{oo2} 
    W(t) + h^{\prime\prime}(t) + \kappa(t) h^{\prime}(t) = \frac{1}{2}\left\|X(t)\right\|^{2} + f(X(t))-f^{\star} + \left\langle \nabla_{\dot{X}(t)}\mathrm{Log}_{X(t)}z,-\dot{X}(t)\right\rangle  \\ + \left\langle \mathrm{Log}_{X(t)}z,\mathrm{grad}f(X(t))\right\rangle\,. \end{split}\end{align} 
        
    From \eqref{Noidea3}, we have the following curvature-dependent bound for the first term in \eqref{Prelim4},
    \begin{equation}\label{Noidea2}\left\langle \nabla_{\dot{X}(t)}\mathrm{Log}_{X(t)}z,-\dot{X}(t)\right\rangle \leq \zeta\left\|\dot{X}(t)\right\|^{2}.
    \end{equation} 
    Using geodesic convexity of \(f\) and since \(f(z)=f^{\star}\), we can rearrange \eqref{SJ2} as
    \begin{align*} \left\langle \mathrm{Log}_{X(t)}z,\mathrm{grad}f\left(X(t)\right)\right\rangle \leq f^{\star} - f\left(X(t)\right) \,,\end{align*} 
    
    and combined with \eqref{Noidea2}, we obtain 
    \begin{align}\begin{split} \label{eqn17} W(t) + h^{\prime\prime}(t) + \kappa(t) h^{\prime}(t) &\leq  \frac{1}{2}\left\|\dot{X}(t)\right\|^{2} + \left(f(X(t))-f^{\star}\right) + \zeta\left\|\dot{X}(t)\right\|^{2}+\left(f^{\star}-f(X(t))\right)  \\ &=  \left(\frac{1 + 2\zeta}{2}\right)\left\|\dot{X}(t)\right\|^{2} =  \frac{\delta}{2}\left\|\dot{X}(t)\right\|^{2}\,, \end{split}\end{align} 

    where \(\delta\) is defined in \eqref{Noidea1}. 
    
    The following calculation 
    \begin{align}\begin{split}\label{1111} W^{\prime}(t)&=\left\langle\nabla \dot{X}(t),\dot{X}(t)\right\rangle + \left\langle \mathrm{grad}f(X(t)),\dot{X}(t),\right\rangle \\&= \left\langle-\kappa(t)\dot{X}(t)-\mathrm{grad}f(X(t)),\dot{X}(t)\right\rangle + \left\langle \mathrm{grad}f(X(t)),\dot{X}(t)\right\rangle \\ &= -\kappa(t)\left\langle\dot{X}(t),\dot{X}(t)\right\rangle\\ &= -\kappa(t)\left\|\dot{X}(t)\right\|^{2}\,, \end{split}\end{align} 
    
    shows that \({W(t)}\) is a non-increasing function.
    
    Now multiply \eqref{eqn17} by \(t\), use \(\kappa(t)=\frac{\alpha}{t}\) and rearrange to obtain 
    \begin{align}\begin{split} \label{teja1} tW(t) &\leq t\frac{\delta}{2}\left\|\dot{X}(t)\right\|^{2} -th^{\prime\prime}(t) - t\kappa(t) h^{\prime}(t) \\ &=  t\frac{\delta}{2}\left\|\dot{X}(t)\right\|^{2} -th^{\prime\prime}(t) - \alpha h^{\prime}(t)\,.\end{split}\end{align} 
    
    Now, \(t\frac{\delta}{2}\left\|\dot{X}(t)\right\|^{2}\) can be written as 
    \begin{align}\label{C10}\frac{\delta}{2} t\left\|\dot{X}(t)\right\|^{2} = \frac{\delta}{2} \frac{t^{2}}{t}\left\|\dot{X}(t)\right\|^{2} = \frac{\delta}{2\alpha}t^{2}\kappa(t)\left\|\dot{X}(t)\right\|^{2} =\frac{\delta}{2\alpha}\left(2tW(t) - (t^{2}W(t))^{\prime}\right)\,,\end{align} 
    
    where the last equality follows from \eqref{1111}. Substituting this in \eqref{teja1} and rearranging yields 
    \begin{align}\label{EQ4} \left(1-\frac{\delta}{\alpha}\right)tW(t) + \left(\frac{\delta}{2\alpha}\right)\left(t^2W(t)\right)^{\prime} \leq -th^{\prime\prime}(t) -\alpha h^{\prime}(t)\,. \end{align}  
    
    Integrating \eqref{EQ4} over \([t_{0},t]\), we obtain 
    \begin{equation}\label{Y1} \left(1-\frac{\delta}{\alpha}\right)\int_{t_{0}}^{t}sW(s)\mathrm{d}s + \left(\frac{\delta}{2\alpha}\right)(t^2W(t)) \leq C_{0} - th^{\prime}(t) + (1-\alpha) h(t)\,,\end{equation} 
    
    where \(C_{0}\coloneqq \frac{\delta}{2\alpha}\left(t_{0}^2W(t_{0})\right)+t_{0}h^{\prime}(t_{0})+(\alpha-1)h(t_{0})\)\,. 
    
    Using \eqref{Prelim3} and applying the Cauchy--Schwarz inequality on the tangent space at \(X(t)\) provides 
    \begin{equation}\label{Teja2} t\left|h^{\prime}(t)\right|\leq t\left\|\mathrm{Log}_{X(t)}z\right\|\,\,\left\|\dot{X}(t)\right\|\,.\end{equation} 
    
    From the definition of \(W(t)\), we have \(\left\|\dot{X}(t)\right\|\leq \sqrt{2W(t)}\)\,. 
    
    Combined with observations made in \eqref{Teja2} and using the fact that \(h(t) = \frac{1}{2}\left\|\mathrm{Log}_{X(t)}z\right\|^{2}\) \(\left(\text{since}\,d(x,y)=\left\|\mathrm{Log}_{x}y\right\|_{x}\right)\), we obtain \(t|h^{\prime}(t)|\leq 2\sqrt{t^{2}W(t)}\sqrt{h(t)}\) and therefore
    \begin{align*}-th^{\prime}(t) \leq  2\sqrt{t^{2}W(t)}\sqrt{h(t)}\,. \end{align*} 
    
    Use this in \eqref{Y1} to arrive at
    \begin{align*}\left(1-\frac{\delta}{\alpha}\right)\int_{t_{0}}^{t}sW(s)\mathrm{d}s + \left(\frac{\delta}{2\alpha}\right)(t^2W(t)) \leq C_{0} + 2\sqrt{t^{2}W(t)}\sqrt{h(t)} - (\alpha-1)h(t)\,.\end{align*} 
    
    Use the inequality \(-ax^{2}+bx\leq \frac{b^2}{4a},\,a>0,\,b\in \mathbb{R}\) with \(a=(\alpha-1)\), \(b = 2\sqrt{t^{2}W(t)}\) and \(x=\sqrt{h(t)}\) to obtain \begin{equation}\label{eql6} A\int_{t_{0}}^{t}sW(s)\mathrm{d}s + Bt^{2}W(t)\leq C_{0}\,,\end{equation} 
    
    where \(A\coloneqq\left(1-\frac{\delta}{\alpha}\right)\) and \(B\coloneqq\left(\frac{\delta}{2\alpha}-\frac{1}{(\alpha-1)}\right)\)\,. 

    Since \(\alpha>\delta\), for both \(K_{\mathrm{min}}\geq 0\) and \(K_{\mathrm{min}} < 0\) we have \(A,\,B\geq 0\). Thus both the terms in \eqref{eql6} are non-negative and upper bounded by a constant \(C_{0}\). Thus we infer from \eqref{eql6} that
    \begin{align}\label{eqn13} \text{sup}_{t\geq t_{0}}t^{2}W(t)&<\infty \quad \text{and} \quad \\ \label{eqn14} \int_{t_{0}}^{+\infty}sW(s)\mathrm{d}s&<\infty\,.\end{align} 
    
    From \eqref{C10}, we have 
    \begin{equation*}\left(t^{2}W(t)\right)^{\prime} = 2tW(t)-t^{2}\kappa(t)\left\|\dot{X}\right\|^{2} \leq 2tW(t) \,,\end{equation*} 
    
    which combined with \eqref{eqn14} gives us that \(\int_{t_{0}}^{\infty}\left(t^{2}W(t)\right)^{\prime}<\infty\) and therefore by \eqref{eqn13}, \(\lim_{t\to\infty} t^{2}W(t)\)  exists. 
    
    Thus, following the chain of arguments \(\mathrm{(i)-(v)}\) stated in the beginning of the proof, we have our desired result.
    \end{proof}
    
    \subsection{Convergence of trajectory when \(\alpha>\delta\)} 
    We now show that for \(\alpha>\delta\), the solution of \eqref{RODE1} converges to an element in \(\mathrm{argmin}_{\mathcal{M}}f\). 
    \begin{theorem}\label{LeviCivita} Assume \(\alpha>\delta\) in \eqref{RODE1}. Then under Assumptions \ref{Assumption1} and \ref{Assumption2}, there exists some \(\tilde{x}\in \mathrm{argmin}_{\mathcal{M}}f\) such that \(X(t)\to \tilde{x}\) as \(t\rightarrow \infty\).
    \end{theorem}
    
    \begin{proof} We come back to \eqref{eqn1001}. For any \(z\in\mathrm{argmin}_{\mathcal{M}}f\), using geodesic convexity of \(f\), we apply \eqref{SJ2} and the fact that \(f^{\star}-f(X(t)) \leq 0\) to \eqref{eqn1001} and obtain, 
    \begin{align}\label{1002} h^{\prime\prime}(\tau)+\kappa(\tau)h^{\prime}(\tau) &\leq \xi\left(d\left(X(\tau),z\right)\right)\left\|\dot{X}(\tau)\right\|^{2}\leq \delta\left\|\dot{X}(\tau)\right\|^{2} \,,\end{align} 
    
    since \(\xi\left(d\left(X(\tau),z\right)\right)<2\xi\left(d\left(X(\tau),z\right)\right)+1\leq \delta\). 
    
    Multiply both sides of \eqref{1002} by \(e^{\Psi(\tau, t_{0})}\) where \(\Psi(\tau, t_{0})=\int_{t_{0}}^{\tau}\kappa(u)\mathrm{d}u\) is the integrating factor, and integrate from \(t_{0}\,\text{to}\, t\). Using standard integration by parts technique and the Fundamental Theorem of Calculus, we obtain, 
    \begin{equation}\label{eq1003} h^{\prime}(t) \leq e^{-\Psi(t,t_{0})} h^{\prime}(t_{0}) + \int_{t_{0}}^{t}e^{-\Psi(t,\tau)} \delta\left\|\dot{X}(\tau)\right\|^{2}\mathrm{d}\tau\,,\end{equation} 

    where \(-\Psi(t,\tau)=\Psi(\tau,t_{0})-\Psi(t,t_{0})\). 
    
    Since \(\kappa(t) = \frac{\alpha}{t}\) a simple integral evaluation gives 
    \begin{equation}\label{eq1004}  \int_{s}^{\infty} e^{{-\Psi(t,s)}}\mathrm{d}t = \frac{s}{\alpha-1},\quad \forall \,s\geq t_{0}\,.\end{equation}
    
    Further integrating \eqref{eq1003} over \([t_{0},\infty)\), we make use of \eqref{eq1004} to obtain 
    \begin{equation} \label{SJ10} \int_{t_{0}}^{\infty}h^{\prime}(t)\mathrm{d}t \leq \frac{t_{0}}{\alpha-1} \left|h^{\prime}(t_{0})\right| + \int_{t_{0}}^{\infty}\int_{t_{0}}^{t}e^{-\Psi(t,\tau)} \delta\left\|\dot{X}(\tau)\right\|^{2}\mathrm{d}\tau\mathrm{d}t\,.\end{equation} 
     
    Now, upon carefully rearranging the domain of integration and subsequently applyling the Fubini's Theorem for double integrals to switch the order of integration we obtain, 
    \begin{equation} \label{SJ11} \int_{t_{0}}^{\infty}h^{\prime}(t)\mathrm{d}t \leq \frac{t_{0}}{\alpha-1} \left|h^{\prime}(t_{0})\right| + \int_{t_{0}}^{\infty}\int_{\tau}^{\infty}e^{-\Psi(t,\tau)} \delta\left\|\dot{X}(\tau)\right\|^{2}\mathrm{d}t\mathrm{d}\tau\,.\end{equation}

    Using \eqref{eq1004} for the inner integral in \eqref{SJ11} we obtain,
    \begin{equation} \label{eq1005} \int_{t_{0}}^{\infty}h^{\prime}(t)\mathrm{d}t \leq \frac{t_{0}}{\alpha-1} \left|h^{\prime}(t_{0})\right| + \frac{1}{\alpha-1}\int_{t_{0}}^{\infty}\tau \delta\left\|\dot{X}(\tau)\right\|^{2}\mathrm{d}\tau\,.\end{equation} 
    
    Finally, from \eqref{eqn14}, the right side of \eqref{eq1005} is finite and we have that \(\int_{t_{0}}^{\infty} h^{\prime}(t)\mathrm{d}t <\infty\) which implies \(\lim_{t\to\infty}h(t)\) exists. Thus we have, 
    \begin{equation}\label{oo1} \lim_{t\to\infty}d\left(X(t),z\right)\,\text{exists for every}\, z\in \mathrm{argmin}_{\mathcal{M}}f\,.\end{equation} 
    
    Up to this stage, we only know that the limit in \eqref{oo1} exists and could be non-zero. Since the trajectory remains bounded, as a consequence of the Hopf--Rinow Theorem for complete Riemannian manifolds, there exists a subsequence \(X(t_{k})_{k\in\mathbb{N}}\) whose accumulation point is say \(\tilde {x}\in \mathcal{C}\subset\mathcal{M}\) \citep{munier2007steepest}. By continuity of \(f\), \(f(\tilde{x}) = f(\lim_{k\to +\infty} X(t_k)) = \lim_{k\to +\infty} f(X(t_k)) = \lim_{t\to +\infty} f(X(t)) = f^\star\), i.e. \(\tilde{x} \in \mathrm{argmin}_\mathcal{M} f\). Now, since \eqref{oo1} holds for every \(z\in\mathrm{argmin}_{\mathcal{M}}f\), in particular it holds for \(\tilde{x}\). This implies \(d\left(X(t_{k}),\tilde{x}\right)\to 0\). However, by uniqueness of limit for the function, \(d\left(X(t),\tilde{x}\right)\to 0\) and therefore \(X(t)\rightarrow \tilde{x}\) (we emphasize that this conclusion is independent of the choice of subsequence \(t_{k}\)).

    \end{proof} 
    
    \subsection{Convergence rate for the sub-critical case \(0<\alpha\leq\delta\)}
    
    In this section we analyze continuous-time dynamics for \eqref{RODE1} for \(0<\alpha\leq \delta\). In the Hilbert space, this has been analyzed in \cite{attouch2019rate} in the continuous-time setting while \cite{apidopoulos2020convergence} analyzed this in the discrete setting. We obtain a similar result for a Riemannian manifold with lower bounded sectional curvature \(K_{\mathrm{min}}\). 
    
    \begin{theorem}\label{Fermat}Assume \(0<\alpha\leq\delta\) in \eqref{RODE1}. Then under Assumptions \ref{Assumption1} and \ref{Assumption2}, any solution \(X\) of \eqref{RODE1} satisfies 
    \begin{equation*}f(X(t))-\mathrm{min}_{\mathcal{M}} f =O\left(\frac{1}{t^{\frac{2\alpha}{\delta}}}\right)\,.\end{equation*}
    \end{theorem}
    
    \begin{proof} We fix a \(z\in \mathrm{argmin}_{\mathcal{M}}f\), define \(f^{\star}\coloneqq \mathrm{min}_{\mathcal{M}}f\) and consider the function \(\map{W}{[t_{0},\infty)}{[0,\infty)}\) given as 
    \begin{equation}\label{1eq1} W(t)= A(t)+B(t)+C(t),\end{equation} 
    where 
    \begin{equation*}\begin{split} A(t) &\coloneqq t^{2p}\left(f(X(t))-f^{\star}\right), \\ B(t)&\coloneqq\frac{1}{2}\left\|-\lambda(t)\left(\mathrm{Log}_{X(t)}z\right)+t^{p}\dot X(t)\right\|^{2}\,, \\ C(t)&\coloneqq =\frac{\eta(t)}{2}d\left(X(t),z\right)^{2}=\frac{\eta(t)}{2}\left\|\mathrm{Log}_{X(t)}z\right\|^{2}\,,\end{split}\end{equation*} 

    and \(p\) is a positive real number, \(\lambda\) and \(\eta\) are positive functions that will be chosen appropriately so as to make the energy function \(W(t)\) non-increasing. 
    
    Then their derivatives are given as
    \begin{align*}\begin{split} A^{\prime}(t) &= 2pt^{2p-1}\left[f(X(t))-f^{\star}\right] + t^{2p}\left\langle \mathrm{grad} f(X(t)),\dot X(t)\right\rangle\,,\\ B^{\prime}(t)&=\left\langle -\lambda(t)\mathrm{Log}_{X(t)}z+t^{p}\dot{X}(t), -\dot{\lambda}(t)\mathrm{Log}_{X(t)}z-\lambda(t)\nabla_{\dot{X}(t)} \mathrm{Log}_{X(t)}z+pt^{p-1}\dot{X}(t)+t^{p}\nabla\dot{X}(t)\right\rangle\,,\\ C^{\prime}(t)&= \frac{\dot{\eta}(t)}{2}\left\|\mathrm{Log}_{X(t)}z\right\|^{2} + \eta(t)\left\langle -\dot{X}(t),\mathrm{Log}_{X(t)}z\right\rangle\,.\end{split}\end{align*}

    We make use of \eqref{RODE1} and write \(B^{\prime}(t)\) as
    \begin{align*}\begin{split}
    B^{\prime}(t)= \Big\langle -\lambda(t)\mathrm{Log}_{X(t)}z+t^{p}\dot{X}(t), -\dot{\lambda}(t)\mathrm{Log}_{X(t)}z-\lambda(t)\nabla_{\dot{X}(t)}\mathrm{Log}_{X(t)}z+pt^{p-1}\dot{X}(t) \\+t^{p-1}(-\alpha)\dot{X}(t)+t^{p}(-\mathrm{grad} f(X(t)))\Big\rangle.\end{split}\end{align*} 
    
    Adding the derivatives of \(A\), \(B\) and \(C\) gives an expression for \({W}^{\prime}(t)\). In order to avoid clutter, we avoid writing the complete expression for \({W}^{\prime}(t)\). Instead, we make some observations about certain terms that appear in that expression that allows us to find an upper bound for \({W}^{\prime}(t)\). 

    Since \(d(X(t),z)^{2}=\left\|\mathrm{Log}_{X(t)}z\right\|^{2} = \left\langle\mathrm{Log}_{X(t)}z,\mathrm{Log}_{X(t)}z\right\rangle\), we have
    \begin{equation*} \frac{1}{2}\frac{\mathrm{d}}{\mathrm{d}t}d\left(X(t),z\right)^{2} = \frac{1}{2}\frac{\mathrm{d}}{\mathrm{d}t}\left\|\mathrm{Log}_{X(t)}z\right\| = \frac{1}{2}\frac{\mathrm{d}}{\mathrm{d}t} \left\langle\mathrm{Log}_{X(t)}z,\mathrm{Log}_{X(t)}z\right\rangle\,.\end{equation*}
    
    Using properties of covariant derivatives, see \cite{Tu1}[Theorem 13.2], we obtain, 
    \begin{equation}\label{abc2} \left\langle \mathrm{Log}_{X(t)}z,\nabla_{\dot{X}(t)}\mathrm{Log}_{X(t)}z\right\rangle = \frac{1}{2}\frac{\mathrm{d}}{\mathrm{d}t}d\left(X(t),z\right)^{2} = \left\langle \mathrm{Log}_{X(t)}z,-\dot{X}(t)\right\rangle\,. \end{equation} 

    Using \eqref{SJ2}, \eqref{Noidea3} and \eqref{abc2} in the expression for the derivative of \(W^{\prime}(t)\) gives us the following bound 
    \begin{multline}\label{1eq11} W^{\prime}(t)\leq t^{p}\left[2pt^{p-1}-\lambda(t)\right]\left(f(X(t)-f^{\star}\right) \\+ \left[\eta(t) + t^{p}\dot{\lambda}(t)-\lambda(t)(\alpha-p)t^{p-1}+\lambda(t)^2\right]\left\langle \mathrm{Log}_{X(t)}z,-\dot{X}(t)\right\rangle \\ -t^{p}\left[(\alpha-p)t^{p-1}-\lambda(t)\zeta\right]\left\|\dot{X}(t)\right\|^2 + \left[\lambda(t)\dot{\lambda}(t)+\frac{\dot{\eta}(t)}{2}\right]\left\|\mathrm{Log}_{X(t)}z\right\|^{2}.
    \end{multline} 
    
    We choose \(\lambda\) and \(\eta\) so as to make the first two terms of \eqref{1eq11} zero. This gives 
    \begin{align}\label{C4} \lambda(t) = 2pt^{p-1} \quad \text{and} \quad \eta(t) = 2p(\alpha-4p+1)t^{2p-2}\,.\end{align} 
    
    To impose that \(\eta\) is non-negative, we impose the condition 
    \begin{equation}\label{C1} \alpha\geq 4p-1\,.\end{equation} 
    
    For \(W\) to be a non-increasing function, we will impose the condition that \((\alpha-p)t^{p-1}-\lambda(t)\zeta\geq 0\) or equivalently 
    \begin{equation}\label{C2} \alpha\geq (2\zeta+1)p =\delta p\,,\end{equation} 
    
    where \(\delta\) is as defined in \eqref{Noidea1}. For the choice of \(\lambda\) and \(\eta\) as in \eqref{C4}, we have 
    \begin{equation*} \lambda(t)\dot{\lambda}(t)+\frac{\dot{\eta}(t)}{2}= -2p(1-p)(\alpha-2p+1)t^{2p-3}\,.\end{equation*} 
    
    which is non-positive if we impose the condition 
    \begin{equation}\label{C5} p\leq 1\,.\end{equation} 
    
    Thus, if we choose \(p=\mathrm{min}\left(1,\frac{\alpha}{\delta},\frac{\alpha+1}{4}\right)\), then conditions \eqref{C1}, \eqref{C2} and \eqref{C5} are satisfied. This implies that \(W\) is a non-negative, non-increasing function associated with \eqref{RODE1}. As a consequence, we obtain
    \begin{equation}\label{C8} t^{2p}\left(f(X(t))-f^{\star}\right)\leq W(t)\leq W(t_{0})\,. \end{equation} 
    
    Now for the case when \(K_{\mathrm{min}} \geq 0\), by definition \(\delta=3\) and hence \(0<\alpha \leq \delta\) implies \(p=\frac{\alpha}{\delta}\). For the case when \(K_{\mathrm{min}}<0\) , we have \(\delta>3\). As a result, \(0<\alpha\leq\delta\) includes the case when \(0<\alpha<3\) and the case when \(3\leq\alpha \leq \delta\). If \(0<\alpha<3\)  then \(1>\frac{\alpha+1}{4}>\frac{\alpha}{3}>\frac{\alpha}{\delta}\) and when \(3\leq\alpha \leq\delta\) then \(\frac{\alpha+1}{4} \geq 1 \geq \frac{\alpha}{\delta}\). Thus, in general \(p=\frac{\alpha}{\delta}\). Combining this with \eqref{C8} we conclude the statement. 
    
    \end{proof} 

    \begin{corollary} As a corollary, we can combine this result with Theorem \ref{pythagoras} and obtain the consolidated rate of convergence for \(\alpha>0\) as 
    \begin{equation*}  f(X(t))-\mathrm{min}_{\mathcal{M}} f =O\left(\frac{1}{t^{p(\alpha)}}\right), \end{equation*} 
      
    where \(p(\alpha)=\mathrm{min}\left(2,\frac{2\alpha}{\delta}\right)\).
    \end{corollary}
    
    \begin{figure}[h!t]
    \centering
    \includegraphics[width=5in]{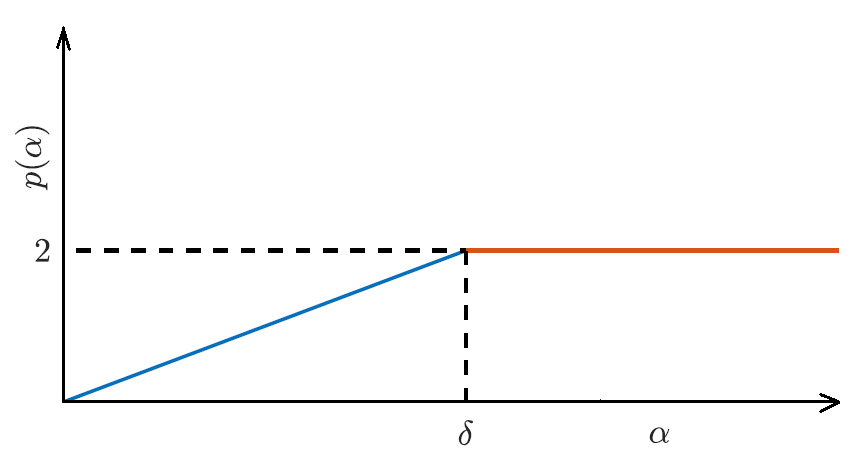}
    \caption{The convergence rate in Corollary 5.4 undergoes a phase change at \(\alpha=\delta\).}
    \label{fig1}
    \end{figure}
   
    \paragraph{Phase Transition.} This result shows a phase transition for convergence rates at \(\alpha=\delta\). For \(\alpha<\delta\), the convergence rate increases linearly with a slope of \(\frac{2}{\delta}\), whereas for \(\alpha\geq\delta\) the convergence rate remains constant. This is shown in Figure \ref{fig1} that corresponds to a similar figure in \cite{attouch2019rate} in the Euclidean case that shows how \(p(\alpha)\) varies as a function of \(\alpha\). The rate of convergence decreases as the value of \(\alpha\) decreases and this is in agreement with the previous works in the literature in the Hilbertian setting. One has then to take \(\alpha\) as large as possible but the rate stagnates at \(o(1/t^2)\) for \(\alpha\) larger than a threshold that depends on the manifold curvature.

    Additionally, we know by the work of \cite{apidopoulos2020convergence} that this rate is optimal for the whole space in the Hilbertian setting. It would be worth investigating whether a similar result holds for a class of manifolds with a given curvature. 

    \subsection{Convergence of trajectories in the sub-critical case \(0<\alpha\leq \delta\)}
    In the Euclidean case, when \(f\) is convex, convergence of solution trajectories of \eqref{ODE1} for the case \(0<\alpha\leq 3\) is still an open problem. However, convergence of the trajectory can be shown by assuming that the convex function has a strong minimum \citep{attouch2019rate}. In the Riemannian setting we can define the notion of a strong minimum of geodesically convex function \(f\) as follows.

    \begin{definition}\label{Def1} A geodesically convex function \(f\) on a Riemannian manifold has a strong minimum if there exists \(\tilde{x}\in\mathrm{argmin}_{\mathcal{M}}f\) and \(\mu>0\) such that for every \(x\in\mathcal{M}\) we have 
    \begin{equation}\label{Noidea4} f(x)\geq f(\tilde{x}) + \frac{\mu}{2}d\left(x,\tilde{x}\right)^{2}\,.\end{equation}
    \end{definition}

    As a result, the minimizer is actually unique. This is true in particular for geodesically strongly convex functions, for example, the Karcher-mean objective (see Section \ref{sec::exp} for details). 
    \begin{theorem}\label{Chandrama} Assume \(0<\alpha\leq \delta\) in \eqref{RODE1}. Then under Assumptions \ref{Assumption1} and \ref{Assumption2}, for a geodesically convex function that admits a strong minimum \(\tilde{x}\), any solution \(X\) of \eqref{RODE1} converges to \(\tilde{x}\) with the rate 
    \begin{equation} d\left(X(t),\tilde{x}\right)^{2} = O\left(\frac{1}{t^{\frac{2\alpha}{\delta}}}\right). \end{equation}  
    \end{theorem}
    \begin{proof} We combine Definition \ref{Def1} with Theorem \ref{Fermat} to obtain
    \begin{equation*} d\left(X(t),\tilde{x}\right)^{2}\leq \frac{2}{\mu}\left(f(X(t))-f(\tilde{x})\right)\,, \end{equation*}
    and the result follows.
    \end{proof}
    
    \section{Numerical Experiments}\label{sec::exp}

    We provide computational evidence for the theoretical guarantees in Section \ref{MR}. In particular we would like to verify faster convergence for increasing values of \(\alpha\leq \delta\) (cf. Theorem \ref{Fermat}) and the little-o rate of convergence for \(\alpha>\delta\) (cf. Theorem \ref{pythagoras}). We consider some standard optimization problems on Riemannian manifolds of positive and negative curvature. For the positive curvature, we consider the maximum eigenvalue problem and for the negative curvature we consider the Karcher-mean problem. These problems have also been considered by \cite{sra2015conic, ferreira2019gradient, alimisis2020continuous}. We integrate the System \ref{RODE1} by employing a semi-implicit discretization as in \cite{su2014differential} and \cite{alimisis2020continuous}. A description of semi-implicit discretization can be found in the Appendix \ref{popo1}.  
    
    In order to demonstrate little-o rate of convergence for objective values, we study the progress of the term \(t^{2}\left(f(X(t))-f^{\star}\right)\), where \(X(t)\) is obtained from the semi-implicit solver while \(f^{\star}\) is a benchmark value obtained from standard libraries in Matlab. For the eigenvalue problem, \(f^{\star}\) is obtained from the Matlab eigenvalue solver whereas for the Karcher-mean problem, \(f^{\star}\) is obtained from the Manopt library \citep{boumal2014manopt}.  For \(\alpha>\delta\), theoretically we expect \(\lim_{t\to\infty} t^{2}\big[f\left(X(t)\right) - f^{\star}\big]=0\). Due to the limitations posed by finite machine precision, a fundamental difficulty in verifying little-o rate of convergence computationally is that the difference \(f(X(t))-f^{\star}\) stagnates beyond a certain stage. This allows \(t^{2}\) to overcompensate and eventually causes their product to grow. As a result, we compute the product till the difference \(f(X(t))-f^{\star}\) is within a tolerance of \(10^{-12}\).    

    \begin{figure}[h!t]
    \centering
    \includegraphics[width=5.25in]{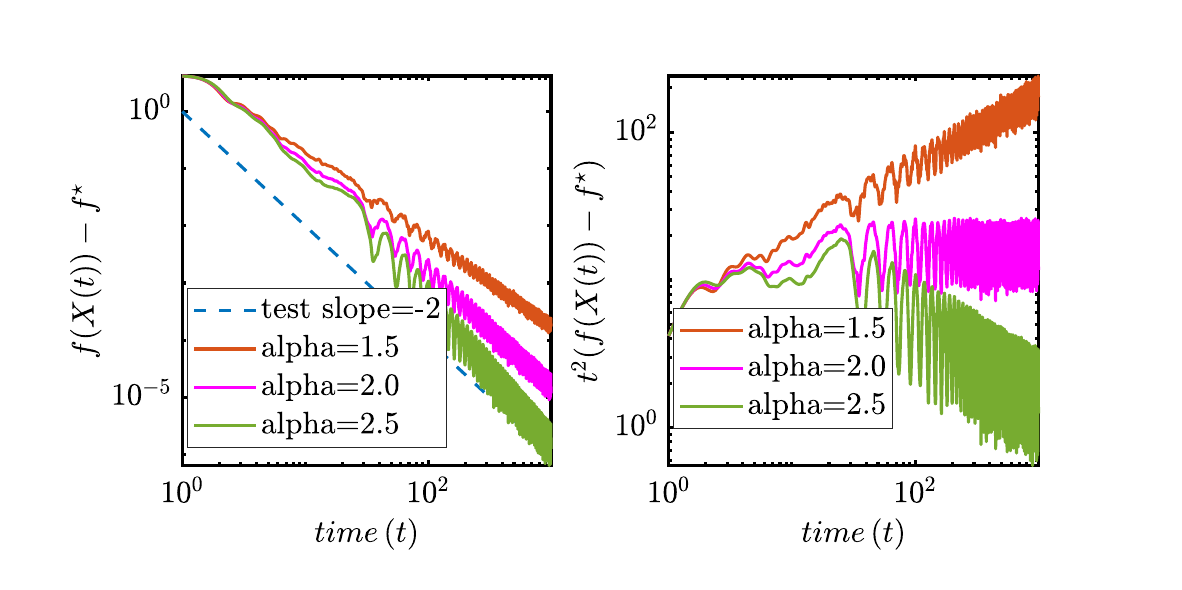}
    \caption{Convergence plots for the max--eigenvalue problem for \(0<\alpha<\delta\).}
    \label{FinPlt1}
    \end{figure}
    
    \begin{figure}[h!t]
    \centering
    \includegraphics[width=5.25in]{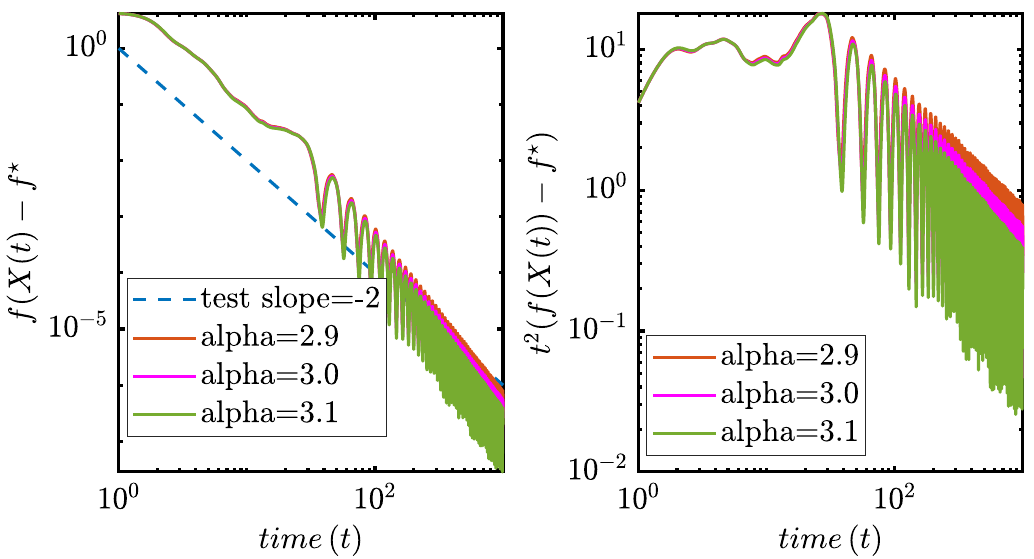}
    \caption{Convergence plots for the max--eigenvalue problem--Transition.}
    \label{FinPlt2}
    \end{figure}

    \begin{figure}[htbp]
    \centering
    \includegraphics[width=5.25in]{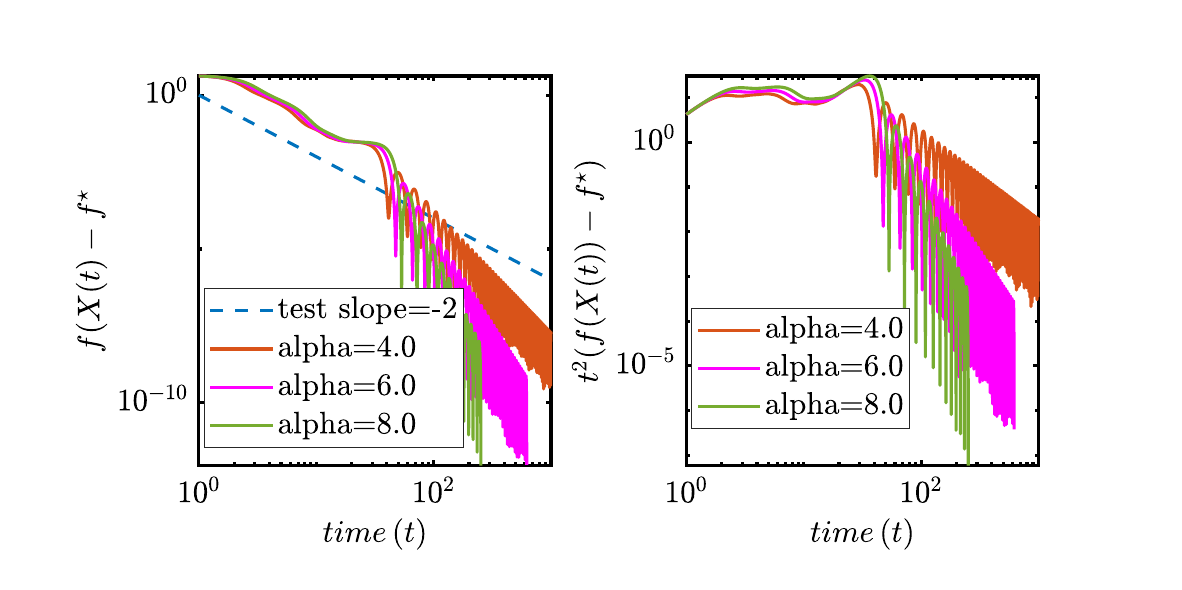}
    \caption{Convergence plots for the max--eigenvalue problem for \(\alpha>\delta\).}
    \label{FinPlt3}
    \end{figure}
    
    \paragraph{Maximum Eigenvalue Problem.} This problem aims to find the maximum eigenvalue of a symmetric positive semi-definite matrix of large condition number. This is accomplished by minimizing the negative of the Rayleigh quotient over the hemisphere. The problem is stated as follows  
    \begin{equation*} \mathrm{min}_{x\in\mathbb{S}}-0.5x^{\top}Ax\,, \end{equation*} 
    
    where \(\mathbb{S}\subset \mathbb{R}^{n}\) is the unit hemisphere. The unit hemisphere has a constant positive curvature \(K_{\mathrm{min}}=1\) and hence \(\delta=3\). We refer the reader to Appendix \ref{popo2} for expressions of exponential map, Riemannian gradient and parallel transport on the sphere.

    We generate the problem instance based on \cite{alimisis2020continuous}. For the experiment, a matrix with high condition number is generated by the formula \(A=\frac{1}{\beta}G^{\top}G\), where \(G\in \mathbb{R}^{m\times n}\,,m>n\), is a random matrix with normally distributed entries with zero mean and variance one. We choose \(m=1000\), \(n=2500\) and \(\beta=1000\). We perform the experiment for different values of \(\alpha\) with \(\alpha<\delta\) and \(\alpha\geq\delta\). Since \(\delta=3\), we perform experiments for \(\alpha=\{1.5,\,2.0,\,2.5,\,2.9,\,3.0,\,3.1,\,4.0,\,6.0,\,8.0\}\). The system is integrated for a length of time \(T=1000\) with step size \(\Delta t=0.1\).  

    The results of numerical experiments are shown in Figures \ref{FinPlt1}, \ref{FinPlt2} and \ref{FinPlt3} where we plot the progress of \(f(X(t))-f^{\star}\) and \(t^{2}\left(f(X(t))-f^{\star}\right)\) against time. While the difference \(f(X(t))-f^{\star}\) tends to zero for all the choices of \(\alpha\), we observe an improving convergence rate for increasing values of \(\alpha\). From Figure \ref{FinPlt1}, it is clear that the product \(t^{2}\left(f(X(t))-f^{\star}\right)\) does not show any decay for \(\alpha=1.5\) and \(\alpha=2.0\). However, from Figure \ref{FinPlt2} we observe that close to \(\delta=3\), the product shows a decreasing trend for \(\alpha=2.5\) and \(\alpha=2.95\) and we see a transition to a convergence rate faster than \(O\left(\frac{1}{t^{2}}\right)\). Finally, from Figure \ref{FinPlt3} for values of \(\alpha>3\) we clearly observe little-o convergence rate. 
    
    \paragraph{Karcher-mean Problem.} This problem aims to find the symmetric positive definite matrix whose sum of squares of distances from a given set of symmetric positive definite matrices is the least. The problem can be posed as a Riemannian optimization problem on the manifold of symmetric positive definite matrices (SPD-manifold) with the affine--invariant metric and is a Hadamard manifold with sectional curvature \(K\in\left[-1/2,0\right)\) \citep{criscitiello2023accelerated}.  
    \begin{equation*} \mathrm{min}_{P\in\mathbb{P}^{n}_{++}} \sum_{j=1}^{m}\left\|\mathrm{Logm}\big(P^{-1/2}A_{j}P^{-1/2}\big)\right\|^{2}_{F} \end{equation*}

    \begin{figure}[htbp]
    \centering
    \includegraphics[width=5.25in]{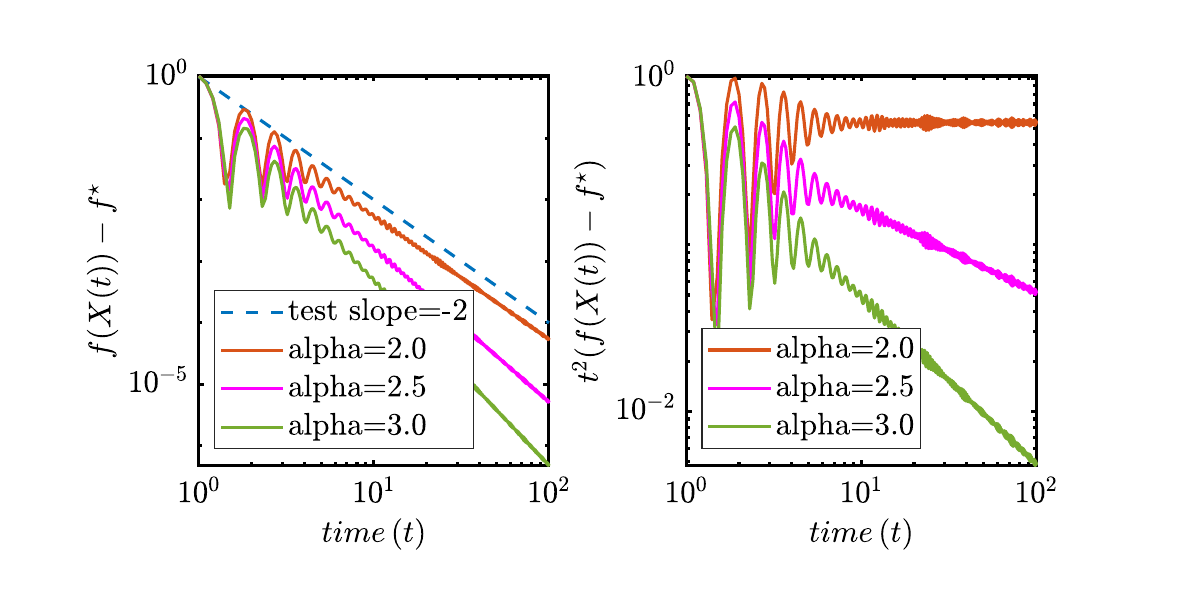}
    \caption{Convergence plots for the Karcher-mean problem for \(0<\alpha<\delta\).}
    \label{FinPlt4}
    \end{figure}

    \begin{figure}[htbp]
    \centering
    \includegraphics[width=5.25in]{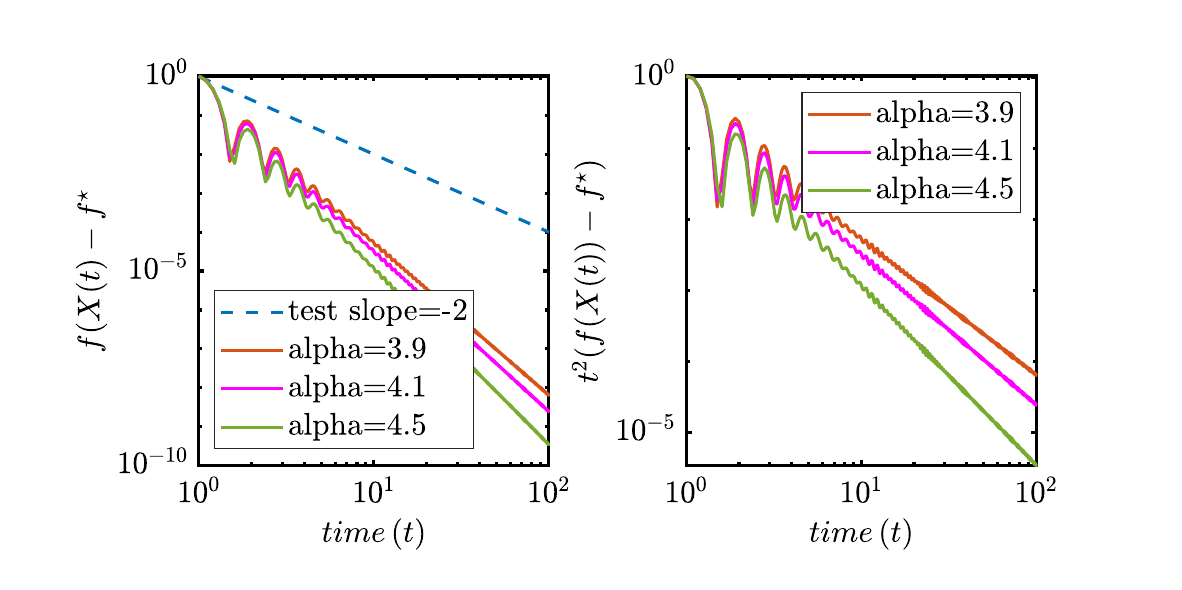}
    \caption{Convergence plots for the Karcher-mean problem--Transition.}
    \label{FinPlt5}
    \end{figure}

    \begin{figure}[htbp]
    \centering
    \includegraphics[width=5.25in]{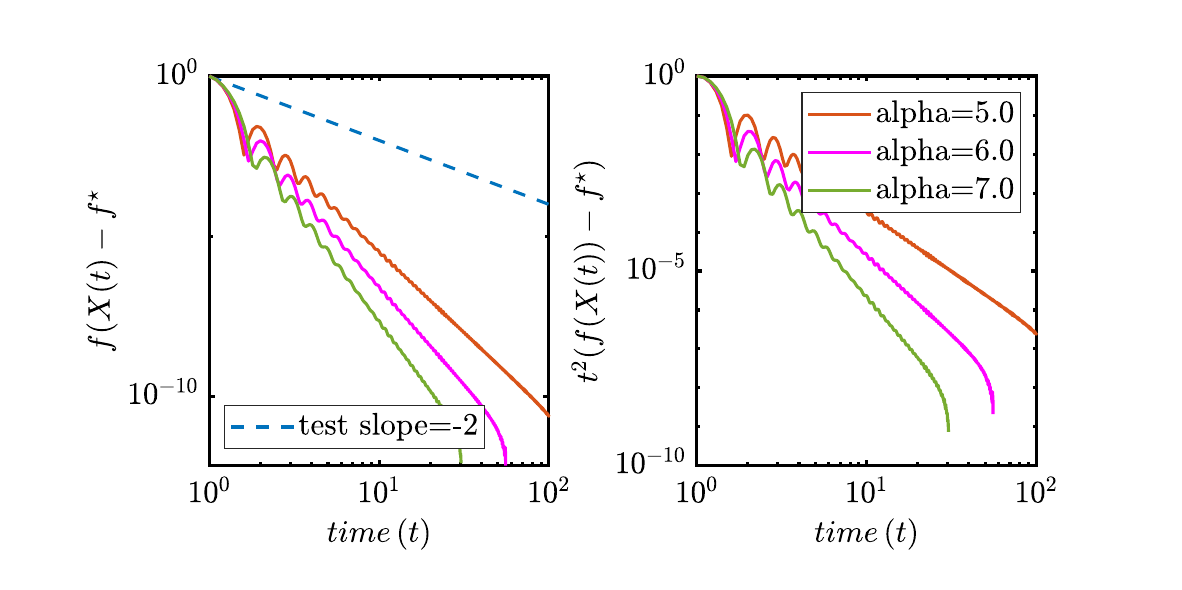}
    \caption{Convergence plots for the Karcher-mean problem for \(\alpha>\delta\).}
    \label{FinPlt6}
    \end{figure}
    
    where \(\mathbb{P}^{n}_{++}\) denotes the SPD-manifold, \(A_{1},\,\ldots,\,A_{m}\in\mathbb{P}^{n}_{++}\), \(\|\cdot\|_{F}\) denotes the Frobenius norm and \(\mathrm{Logm}\) denotes matrix logarithm. The Karcher-mean problem is Euclidean non-convex problem but Riemannian strongly convex \citep{ferreira2019gradient}. This is an important application of Riemannian optimization where a Euclidean non-convex problem can be studied and solved as a geodesically convex problem. The expressions for the exponential map, Riemannian gradient and parallel transport have been given in Appendix \ref{popo2}.
    
    For the experiment, we compute the Karcher-mean of ten randomly generated SPD-matrices \((m=10)\) of size \(n=100\). We employ the strategy proposed in \cite{ferreira2019gradient} to generate SPD-matrices for the experiment and the starting point \(X_{0}\). In order to generate the matrices, for \(j=1,\ldots,m\), we generate random orthonormal matrix \(U_{j}\) and diagonal matrix \(Q_{j}\) with eigenvalues in \((0,100)\). Then \(A_{j}=U_{j}Q_{j}U_{j}^{\top}\in\mathbb{P}^{n}_{++}\). The initial point \(P_{0}\) is given as the explog--geometric mean \(P_{0}=\mathrm{Expm}\left(\frac{1}{m}\sum_{j=1}^{m}\mathrm{Logm}(A_{j})\right)\), where \(\mathrm{Expm}\) denotes matrix exponential. 
    
    We first estimate the value of \(\delta\) for this problem. The diameter \(D\) is estimated as the distance between the optimal \(P\in\mathbb{P}_{++}^{n}\) obtained from the Manopt--solver and the initial point \(X_{0}\). We then choose \(K_{\mathrm{min}}=-0.1\). This gives \(\zeta\approx 1.59\) and \(\delta=2\zeta+1\approx 4.1\). We integrate the system for a length of time \(T=100\) with step size \(\Delta t=0.1\) and for \(\alpha<\delta\), \(\alpha=\delta\) and \(\alpha>\delta\). Since \(\delta\approx 4.1\), we perform experiments for \(\alpha=\{2.0,\,2.5,\,3.0,\,3.9,\,4.1,\,4.5,\,5.0,\,6.0,\,7.0\}\).
    
    The computational experiments agree with our theoretical results. We observe that the convergence is much faster than the previous example on the sphere which is due to the fact that the Karcher-mean problem is a geodesically strongly convex. It is evident from Figures \ref{FinPlt4}, \ref{FinPlt5} and \ref{FinPlt6} that while the difference \(f(X(t))-f^{\star}\) tends to zero for all the choices of \(\alpha\), the convergence is faster for values of \(\alpha\) increasing to \(\delta\). While we do not observe little-o rates for \(\alpha=2.0\), we observe a decreasing trend for \(t^2\left(f(X(t))-f^{\star}\right)\) for \(\alpha=2.5\) and \(\alpha=3\). We observe little-o rates for values of \(\alpha\) close to \(\delta\approx 4.1\) and for values of \(\alpha\) greater than \(\delta\) as evident from Figures \ref{FinPlt5} and \ref{FinPlt6}. The fact that we observe little-o rates for values of \(\alpha\) close to \(\delta\approx 4.1\) is due to the fact that the value of \(\delta\) we have chosen is not a tight estimate and is obtained rather heuristically. 

    Thus, we observe that for experiments performed on manifolds of both the positive as well as the negative curvature, the computational results seem to be in line with our theoretical results.    
    \section{Conclusion}\label{sec::conclusion}
 
    In this work we studied the continuous-time dynamical system to model accelerated first-order optimization algorithms on Riemannian manifolds. We have closed gaps in the convergence guarantees between the Euclidean setting and the Riemannian setting. In particular, corresponding to \(\alpha>3\) in the Euclidean setting, we show that the convergence rate for objective values is \(o\left(\frac{1}{t^{2}}\right)\) for \(\alpha>\delta\) on Riemannian manifolds. This rate is faster than the previously known rate of \(O\left(\frac{1}{t^{2}}\right)\) shown by \cite{alimisis2020continuous}. In the same setting, we also show the convergence of trajectory to an element in the set of minimizers of the objective function. We analyze the dynamical system in the sub-critical case \(0<\alpha\leq \delta\) and provide convergence rate for objective values. In this sub-critical case, we show the convergence of trajectory to a minimizer that satisfies the strong minimization property. We perform computational experiments that confirm the theoretical results.

    We end this paper with some closing comments on some aspects of accelerated dynamics on Riemannian manifolds that we encountered during this work. We note that the accelerated dynamical system that we have considered cannot be studied on some rather standard manifolds like the sphere as the exponential map is not invertible on the sphere. However, this is not necessarily a limitation. This is because the sphere is a compact manifold and the only convex function that can be defined on a compact manifold is the constant function. But in general, it would be worth investigating if the assumption that the exponential map is a diffeomorphism can be relaxed. 
    
    Another important aspect is the analysis of various discretizations of the accelerated dynamical system. While we have considered the semi-implicit discretization, some other discretizations like the explicit discretization are worth considering. It would be interesting to study whether these discretizations are equivalent to the proposed first-order accelerated algorithms on Riemannian manifolds. Additionally, it is worth exploring the tightness of the value of \(\delta\) as it is evident that the behavior resembling little-o rate of convergence starts to appear for values lesser than the value of \(\delta\) considered. Finally, the problem of convergence of trajectories for a convex function in the case \(\alpha=\delta\) is an open question even in the Euclidean case \((\alpha=3)\) for dimensions higher than one. 

	\section*{Acknowledgment}
	T. Natu, C. Castera, J. Fadili and P. Ochs are supported by the ANR-DFG joint project TRINOM-DS under the numbers ANR-20-CE92-0037-01 and OC150/5-1.

 	\newpage
	\appendix
	\begin{flushleft}
		{\LARGE \textbf{Appendices}}
	\end{flushleft}

 
    \section{Supplementaries for Theorem \ref{pythagoras}}
    \subsection{Lemma} 
    We draw attention to a rather simple and standard result that eventually allows us to prove the little-o rate. 
    
    \begin{lemma}\label{lema1} Consider two non-negative functions \(a\) and \(b\) such that \(a(t)b(t)\) has a limit \(m\geq 0\) as \(t\rightarrow\infty\). Then, if the function b(t) is integrable and \(\frac{1}{a(t)}\) is non-integrable, this means that \(m=0\).
    \end{lemma}
    \begin{proof} If the limit of \(a(t)b(t)\) is not zero, then \(b(t)\geq \frac{\tilde{m}}{a(t)}\), for some \(\tilde{m}\in (0,m)\) and sufficiently large \(t\), which contradicts the fact that \(b(t)\) is integrable. 
    \end{proof}

    \subsection{Derivatives of \(h\)}\label{TH1}
    The expressions for the derivatives of \(h(t)\) in \eqref{Prelim3} and \eqref{Prelim4} may appear rather abstract, therefore it is worthwhile to draw parallels with the Euclidean case. 
    
    Suppose \(\mathcal{M}=\mathbb{R}^{n}\). Then the expressions for the derivatives of \(h\) can be calculated by applying the chain rule and are given as 
    \begin{align}\label{Prelim1}h^{\prime}(t) &= \left\langle X(t)-z,\dot{X}(t)\right\rangle\,,\\\label{Prelim5} h^{\prime\prime}(t) &= \left\langle \dot{X}(t),\dot{X}(t)\right\rangle + \left\langle X(t)-z,\ddot{X}(t)\right\rangle = \left\|\dot{X}(t)\right\|^{2} + \left\langle X(t)-z,\ddot{X}(t)\right\rangle\,. \end{align} 

    In \eqref{Prelim1}, the term \(X(t)-z\) corresponds to the term \(\mathrm{Log}_{X(t)}z\) in \eqref{Prelim3}. Observe that in the Euclidean setting, \(K_{\mathrm{min}}=K_{\mathrm{max}}=0\), thus the bounds in \eqref{OMG1} are satisfied with an equality and therefore the first term in \eqref{Prelim4} equals the first term in \eqref{Prelim5}. Similarly, in the second term in \eqref{Prelim4}, the covariant derivative \(\nabla \dot{X}\) corresponds to the term \(\ddot{X}\) in \eqref{Prelim5}.

    \section{Computational Results}
    \subsection{Semi-Implicit Discretization} \label{popo1}
    In order to discretize the dynamical system, we observe that the second-order system in \eqref{RODE1} can be written in an equivalent form as a first-order system in phase space by introducing a new variable \(V\) for velocity as
    \begin{align} \label{OK1}\begin{split}\dot{X} &= V\,,\\ \nabla V &= -\frac{\alpha}{t} V - \mathrm{grad}(f(X(t)))\,.\end{split} \end{align} 

    The semi-implicit discretization is performed by taking an explicit step in the \(V\) variable using \(V_{k}\) at the point \(X_{k}\) to obtain \(\tilde{V}_{k+1}\). The position variable is updated implicitly by applying the exponential map on \(\tilde{V}_{k+1}\) at the point \(X_{k}\) to obtain \(X_{k+1}\). Then, the updated velocity \(\tilde{V}_{k+1}\) is parallel transported to \(X_{k+1}\) to obtain \(V_{k+1}\). For the system of differential equations \eqref{OK1}, this is summarized as 
    \begin{align*}\tilde{V}_{k+1} &= \left(1-\alpha\frac{\Delta t}{t_{k}}\right)V_{k} - \mathrm{grad}(f(X_{k}))\Delta t\,,\\ X_{k+1} &= \mathrm{Exp}_{X_{k}}(\tilde{V}_{k+1}\Delta t)\,, \\ V_{k+1}&= \Gamma_{X_{k}}^{X_{k+1}}\tilde{V}_{k+1}.\end{align*}
    
    where \(\Gamma_{X_{k}}^{X_{k+1}}\) denotes the parallel transport of \(\tilde{V}_{k+1}\) from the point \(X_{k}\) to \(X_{k+1}\) and \(\Delta t\) is the length of the time step.
    
    \subsection{Manifold Toolbox}\label{popo2}
    The expressions for exponential map, Riemannian gradient and parallel transport on a sphere can be found in \cite{boumal_2023} or \cite{absil2008optimization} and are summarized as 
    \begin{align*} \mathrm{grad} f(x) &= (I - xx^{\top})(-Ax) \,\,(\text{Riemannian Gradient})\,,\\\ \mathrm{Exp}_{x}(v) &= \mathrm{cos}(\left\|v\right\|)x + \mathrm{sin}(\left\|v\right\|)\frac{v}{\left\|v\right\|}\,\,(\text{Exponential Map})\,, \\ \Gamma_{x}^{y}(v) &= v - (xx^{\top})v \,\,(\text{Parallel Transport})\,.\end{align*}
    where \(I\) is an identity matrix.

    The expressions for exponential map, Riemannian gradient and parallel transport on the SPD-manifold can be found in \cite{ferreira2019gradient, gutman2023coordinate, AxenBaranBergmannRzecki:2023} and are summarized as follows  
    \begin{align*} \mathrm{grad} f(P) &= \sum_{i=1}^{m}X^{\frac{1}{2}}\mathrm{Logm}(P^{\frac{1}{2}}A_{i}^{-1}P^{\frac{1}{2}})P^{\frac{1}{2}}\,\,(\text{Riemannian Gradient})\\\ \mathrm{Exp}_{P}(V) &= P^{\frac{1}{2}}\mathrm{Expm}\left(P^{-\frac{1}{2}}VP^{-\frac{1}{2}}\right)P^{\frac{1}{2}}\,\,(\text{Exponential Map}) \\ \Gamma_{P}^{\tilde{P}}(V) &= (\tilde{P}P^{-1})^{\frac{1}{2}}V(P^{-1}\tilde{P})^{\frac{1}{2}}\,\,(\text{Parallel Transport}).\end{align*} 

    where \(\mathrm{Expm}\) and \(\mathrm{Logm}\) denote the matrix exponential and logarithm. 

    \bibliographystyle{apalike}
    \bibliography{bibliography.bib}
    
\end{document}